\documentclass[a4paper]{amsart}
\title[Schwarz maps associated with $(2,4,4)$ and $(2,3,6)$]
{Schwarz maps associated with the triangle groups $(2,4,4)$ and $(2,3,6)$}
\date{\today}
\usepackage[dvips]{graphicx} 
\usepackage{amssymb}
\usepackage[usenames]{color}

\catcode`@=11
\def\middlebar#1{$\setbox0=\hbox{$\hphantom{\hbox{#1}}$} \dp0=0pt
\hbox{#1}\llap{ \raise 4pt\hbox{$\m@th \underline{\raise 5pt \box0}$}}$}
\catcode`@=12
\newcommand{\C}{\mathbb C}

\newcommand{\N}{\mathbb N}

\renewcommand{\H}{\mathbb H}
\renewcommand{\P}{\mathbb P}
\newcommand{\Q}{\mathbb Q}
\newcommand{\Z}{\mathbb Z}

\newcommand{\CF}{\mathcal F}
\newcommand{\CM}{\mathcal M}
\newcommand{\re}{\mathrm{Re}}
\newcommand{\im}{\mathrm{Im}}

\newcommand{\pa}{\partial}
\newcommand{\ex}{\mathbf{e}}

\renewcommand{\a}{\alpha} 
\renewcommand{\b}{\beta} 
\newcommand{\g}{\gamma}

\newcommand{\h}{\vartheta}

\newcommand{\f}{\varphi}
\newcommand{\z}{\zeta}
\newcommand{\G}{\varGamma}

\newcommand{\w}{\omega}
\newcommand{\ii}{i}
\newcommand{\la}{\langle}
\newcommand{\ra}{\rangle}
\newcommand{\tr}{\;^t}
\newcommand{\pr}{\mathrm{pr}}

\newtheorem{theorem}{Theorem}
\newtheorem{proposition}{Proposition}
\newtheorem{lemma}{Lemma}
\newtheorem{cor}{Corollary}
\newtheorem{fact}{Fact}
\newtheorem{remark}{Remark}

\def\comment#1{ }
\usepackage{amsthm}
\theoremstyle{plain}

\numberwithin{equation}{section}
\numberwithin{figure}{section}

\author{Yuto Koguchi}
\address[Koguchi]{
Product Actuarial Team, 
The Gibraltar Life Insurance Co., Ltd., 
2-13-10 Nagata-cho, 
Tokyo 100-8953, Japan
}
\author{Keiji Matsumoto}
\address[Matsumoto]{
   Department of Mathematics,
   Hokkaido University,
   Sapporo 060-0810, Japan
}
\email{matsu@math.sci.hokudai.ac.jp}
\author{Fuko Seto}
\address[Seto]{
   Aomori Prefectural Hachinohe Kita High School,
   8-3 Machimichi, Ohkubo, 
   Hachinohe 031-0833, Japan
}

\keywords{
Schwarz map, theta function, mean iteration}
\subjclass[2010]{14K25, 33C05}
\begin{document}
\maketitle
\begin{abstract}
We study the Schwarz maps with monodromy groups 
isomorphic to the triangle groups $(2,4,4)$ and $(2,3,6)$ and 
their inverses. 
We apply our formulas to the study of mean iterations.
\end{abstract}
\tableofcontents

\section{Introduction}
The Gauss hypergeometric function $F(\a,\b,\g;z)$ is defined by the 
series 
$$
F(\a,\b,\g;z)=\sum_{n=0}^\infty \frac{(\a,n)(\b,n)}{(\g,n)(1,n)}z^n,
$$
where $z$ is the main variable in the unit disk 
$\mathbb{D}=\{z\in \C\mid |z|<1\}$, $\a,\b,\g$ are parameters 
with $\g\ne 0,-1,-2,\dots$, and $(\a,n)=\a(\a+1)\cdots(\a+n-1)$.
This function admits an integral representation 
\begin{equation}
\label{eq:Euler}
\frac{\G(\g)}{\G(\a)\G(\g-\a)}
\int_1^\infty  t^{\b-\g}(t-z)^{-\b}(t-1)^{\g-\a-1}dt,
\end{equation}
and satisfies the hypergeometric differential equation 
\begin{equation}
\label{eq:HGDE}
\CF(\a,\b,\g):z(1-z)f''(z)+\{\g-(\a+\b+1)z\}f'(z)-\a\b f(z)=0,
\end{equation}
which has only singular points of regular type 
at $z=0,1,\infty$. 
The Schwarz map is defined by the continuation to $X=\C-\{0,1\}$ 
of the ratio of two linearly independent solutions to $\CF(\a,\b,\g)$ 
in a small simply connected domain in $X$. 

In this paper, we study the Schwarz maps for two sets of the parameters 
$$(\a,\b,\g)=\big(\frac{1}{4},0,\frac{1}{2}\big),\quad 
\big(\frac{1}{3},0,\frac{1}{2}\big).$$
The monodromy groups of $\CF(\a,\b,\g)$ for these sets of parameters are  
reducible and isomorphic to the triangle groups 
$(2,4,4)$ and $(2,3,6)$, respectively. 
We give circuit matrices generating these groups 
in Corollary \ref{cor:monod}. 
The images of the Schwarz maps are the quotient of 
the complex torus $E_\ii=\C/(\ii\Z+\Z)$ by the 
the multiplicative group $\la \ii\ra=\{\pm1,\pm\ii\}$
for $(\a,\b,\g)= (\dfrac{1}{4},0,\dfrac{1}{2})$, and 
that of $E_\z=\C/(\z\Z+\Z)$ by $\la \z\ra=\{\pm1,\pm\z,\pm\z^2\}$
for $(\a,\b,\g)= (\dfrac{1}{3},0,\dfrac{1}{2})$, 
where $\ii=\sqrt{-1}$ and $\z=\dfrac{1+\sqrt{3}\ii}{2}$. 
We consider elliptic curves 
$$C_\ii: u^4=t^2(t-1),\quad C_\z:u^6=t^3(t-1),$$
and relate these Schwarz maps and the Abel-Jacobi maps 
$$\jmath_\ii:C_\ii\to E_\ii,\quad \jmath_\z:C_\z\to E_\z$$
defined by incomplete elliptic integrals on $C_\ii$ and on $C_\z$.
We express the inverses of these Schwarz maps in terms of the theta function 
$\h_{a,b}(z,\tau)$ with characteristics $a,b$; 
see Theorem \ref{th:inverse} and Theorem \ref{th:(t,u)-rep}.
As corollaries, some $\h_{a,b}(0,\ii)$ and $\h_{a,b}(0,\z)$ 
are evaluated in terms of the Gamma function. 
We study the pull-back of the $(1+\ii)$-multiple on $E_\ii$ and 
that of the $(1+\z)$-multiple $E_\z$ under the 
corresponding Abel-Jacobi maps. 
We apply these results to the study of mean iterations. We show that  
the formula in Theorem \ref{th:(1+i)-multiple}  
yields a limit formula of a mean iteration 
$$
\lim_{n\to\infty} \overbrace{m\circ \cdots \circ m}^{n}(a,b)
=\dfrac{a}{F\big(\frac{1}{4},\frac{1}{2},\frac{5}{4};1-\frac{b^2}{a^2}
\big)^2},
$$
where $a>b>0$ and 
$$m:(a,b)\mapsto \left(\frac{a+b}{2},\sqrt{\frac{a(a+b)}{2}}\right).
$$
We have a similar result from the $(1+\z)$-multiple formula 
on the elliptic curve $E_\z$ in Theorem \ref{th:(1+z)-multiple}.
It is studied in \cite{HKM} that these limit formulas for mean iterations 
can be obtained from transformation formulas for 
the hypergeometric function in \cite{G}. 
We elucidate a geometric background of these limit formulas
as multiplications on the complex tori $E_\ii$ and $E_\z$.

\section{The Schwarz map}
\subsection{Fundamental system of solutions to $\CF(a,b,c)$}
We define the Schwarz map as the ratio of solutions to 
$\CF(\a,\b,\g)$ 
given by the Euler type integral representations 
\begin{align*}
f_1(x)&=\int_1^x t^{\b-\g}(t-x)^{-\b}(t-1)^{\g-\a}\frac{dt}{t-1},\\
f_2(x)&=\int_1^\infty t^{\b-\g}(t-x)^{-\b}(t-1)^{\g-\a}\frac{dt}{t-1},
\end{align*}
where 
$$0<\re(\a)<\re(\g),\quad  \re(\b)<1.$$
For an element $x$ in $U=\{x\in X\mid\max(|x|,|1-x|)<1,\im(x)>0\}$, 
they can be expressed by 
the hypergeometric series. 
By (\ref{eq:Euler}), 
$$f_2(x)=B(\g-\a,\a)\cdot F(\a,\b,\g;x),$$ 
where $B(*,*)$ denotes the beta function.  
By the variable change 
$$s=\frac{x-1}{t-1},\quad\textrm{i.e.}\quad  t=\frac{s+x-1}{s}, \quad 
dt=-\frac{(x-1)ds}{s^2}$$ 
for the integral representation of $f_1(x)$ and (\ref{eq:Euler}), we have
$$
f_1(x)=e^{\pi\ii(\g-\a)}B(\g-\a,1-\b)(1-x)^{\g-\a-\b}\cdot 
F(\g-\a,\g-\b,\g-\a-\b+1;1-x),
$$
where $\theta_1=\arg x$ and $\theta_2=\arg(1-x)$ belong to 
the open interval $(-\pi/2,\pi/2)$, and 
the arguments of $t$, $t-x$,  $t-1$ on the open segments 
$(1,x)$ and $(1,\infty)$ take values in the intervals in 
Table \ref{tab:argument}.
Here pay your attention to the argument of $t-1$ and 
the orientation of the path integral.

\begin{center}
\begin{table}[bht]
\begin{tabular}[bht]{l|cc}
          & $t\in (x,1)$                        & $t\in (1,\infty)$  \\
\hline
$\arg(t)$ &$[\min(0,\theta_1),\max(0,\theta_1)]$& $0$ \\  
$\arg(t-x)$&$\theta_2$& $[\min(0,\theta_2),\max(0,\theta_2)]$\\
$\arg(t-1)$ &$\pi+\theta_2$ & $0$\\[2mm]
\end{tabular}
\caption{Arguments of $t$, $t-x$ and $t-1$}
\label{tab:argument}
\end{table}
\end{center}
\begin{remark}
\label{rem:binZ}
When $\b=0$, the solution $f_1(x)$  is expressed as
$$
f_1(x)=\frac{e^{\pi\ii(\g-\a)}}{\g-\a}\cdot (1-x)^{\g-\a}\cdot 
F(\g-\a,\g,\g-\a+1;1-x)
$$
for $|x-1|<1$, and the solution $f_2(x)$ reduces to a constant   
$$B(\g-\a,\a)=\frac{\G(\g-\a)\G(\a)}{\G(\g)}.$$

\end{remark}

\subsection{Monodromy representation of $\CF(a,b,c)$}
We take a base point $\dot x$ in $U$.
Let $\CM$ be the monodromy representation of $\CF(\a,\b,\g)$ with respect 
to the base point $\dot x$. It is the homomorphism from 
the fundamental group $\pi_1(X,\dot x)$ to the general linear group 
of the local solution space to $\CF(\a,\b,\g)$ on $U$ arising from the 
analytic continuation along a loop with terminal $\dot x$.
We denote the image of $\ell\in \pi_1(X,\dot x)$ by $\CM_\ell$.
Let $\ell_0$ and $\ell_1$ be a loop starting from $\dot x$ 
turning positively around the point $x=0$ and that around the point $x=1$, 
respectively. Since $\pi_1(X,\dot x)$ is generated by $\ell_0$ and 
$\ell_1$, $\CM$ is determined by $\CM_0=\CM_{\ell_0}$ 
and $\CM_1=\CM_{\ell_1}$. 
By the basis $\tr(f_1(x),f_2(x))$,  
the transformations $\CM_0$ and $\CM_1$ are represented by 
matrices $M_0$ and $M_1$. 
That is, the basis $\tr(f_1(x),f_2(x))$ is transformed into 
$$M_i\begin{pmatrix}f_1(x)\\f_2(x)
\end{pmatrix}
$$ 
by the analytic continuation along the loop $\ell_i$. 
They are expressed by the intersection matrix
$$
H=\begin{pmatrix}
\dfrac{\ex(\g-\a)-\ex(\b)}{\ex(\g-\a)-1} &
\dfrac{-\ex(\g-\a)}{\ex(\g-\a)-1}\\[4mm]
\dfrac{-\ex(\b)+1}{\ex(\g-\a)-1}&
\dfrac{-\ex(\g)+1}{(\ex(\g-\a)-1)(\ex(\a)-1)}
\end{pmatrix}
$$
as in \cite{Ma}.

\begin{proposition}
Suppose that 
$$\a,\quad \a-\g,\quad \b-\g\notin \Z.$$
Then we have 
\begin{align*}
M_0&=\lambda_0 I_2-\frac{\lambda_0-1}{e_2 H\; e_2^*}H\; e_2^*e_2
=\begin{pmatrix}
\ex(-\g) & 1-\ex(-\a)
\\
0 & 1
\end{pmatrix},
\\
M_1&=I_2-\frac{1-\lambda_1}{e_1 H\; e_1^*}H\; e_1^*e_1
=\begin{pmatrix}
\ex(\g-\a-\b) & 0\\
-1+\ex(-\b) & 1
\end{pmatrix},
\end{align*}
where $\lambda_0=\ex(-\g)$, $\lambda_1=\ex(\g-\a-\b)$, 
$$I_2=\begin{pmatrix}
1 & 0 \\ 0 & 1
\end{pmatrix},\quad 
e_1=(1,0),\quad e_2=(0,1),\quad e_1^*=\begin{pmatrix}1 \\ 0\end{pmatrix},
\quad e_2^*=\begin{pmatrix}0 \\ 1\end{pmatrix}.
$$
\end{proposition}
We use this proposition for $\b\in \Z$ with a base change 
$$
\begin{pmatrix}1 & 0\\ 0 & 1-\ex(\a)
\end{pmatrix}
\begin{pmatrix} f_1(x)\\ f_2(x)
\end{pmatrix}.
$$
\begin{cor}
\label{cor:monod}
In this case, $M_0$ and $M_1$ are transformed into 
$$N_0=\begin{pmatrix}
\ex(-\g) & -\ex(-\a)\\
0 & 1
\end{pmatrix},
\quad 
N_1
=\begin{pmatrix}
\ex(\g-\a) & 0\\
0 & 1
\end{pmatrix},
$$
respectively. When $(\a,\b,\g)=(\frac{1}{4},0,\frac{1}{2})$, 
$N_0$, $N_1$, $(N_0N_1)^{-1}$ are 
$$
\begin{pmatrix}-1 & \ii\\ 0 &1 
\end{pmatrix},\quad 
\begin{pmatrix}\ii & 0\\ 0 &1 
\end{pmatrix},\quad 
\begin{pmatrix}\ii & 1\\ 0 &1 
\end{pmatrix}.
$$
The group generated by these matrices is isomorphic to the 
triangle group $(2,4,4)$, and to the semi-direct product 
$\la \ii\ra \ltimes \Z[\ii]$. 
When $(\a,\b,\g)=(\frac{1}{3},0,\frac{1}{2})$, 
$N_0$, $N_1$, $(N_0N_1)^{-1}$ are 
$$
\begin{pmatrix}-1 & \z\\ 0 &1 
\end{pmatrix},\quad 
\begin{pmatrix}\z & 0\\ 0 &1 
\end{pmatrix},\quad 
\begin{pmatrix}\z^2 & 1\\ 0 &1 
\end{pmatrix},\quad \z=\frac{1+\sqrt{3}\ii}{2}.
$$
The group generated by these matrices is isomorphic to the 
triangle group $(2,3,6)$, and to the semi-direct product 
$\la \z\ra \ltimes \Z[\z]$. 

\end{cor}

\section{Theta functions}
\subsection{Basic properties of $\h_{a,b}$}
The theta function with characteristics is defined by 
$$\h_{a,b}(z,\tau)
=\sum_{n\in \Z} \exp(\pi \ii (n+a)^2\tau+2\pi\ii (z+b)),$$
where $z\in \C$ and $\tau\in \H$ are main variables, and 
$a,b$ are rational parameters. 
For a fixed $\tau$, we denote $\h_{a,b}(z,\tau)$ by $\h_{a,b}(z)$.
In this subsection, we collect useful formulas for $\h_{a,b}(z,\tau)$ 
in our study from \cite{I} and \cite{Mu}.

It is easy see that this function satisfies 
\begin{align*}
\h_{a,b}(z,\tau)&=\ex\big(\frac{a^2\tau}{2}+a(z+b)\big)
\h_{0,0}(z+a\tau+b,\tau),\\
\h_{-a,-b}(z,\tau)&=\h_{a,b}(-z,\tau),\\
\h_{a,b}(z+p\tau+q,\tau)&=\ex\big(aq-\frac{p^2\tau}{2}-pz-bp\big)
\h_{a,b}(z,\tau)\\
\h_{a+p,b+q}(z,\tau)&=\ex(aq)\h_{a,b}(z,\tau),\\
\frac{\h_{a,b}(z+c\tau+d,\tau)}{\h_{a',b'}(z+c\tau+d,\tau)}
&=\ex(c(b'-b))\frac{\h_{a+c,b+d}(z,\tau)}{\h_{a'+c,b'+d}(z,\tau)},
\end{align*}
where $p,q\in \Z$ and $a',b'\in \Q$.

It is known that 
$\h_{a,b}(z)=0$ if and only if 
$$\big(-a+p+\frac{1}{2}\big)\tau+\big(-b+q+\frac{1}{2}\big)
\quad (p,q\in \Z),$$
and they are simple zeroes. 
If $(a_1,b_1),\dots,(a_r,b_r)$ and $(a'_1,b'_1),\dots,(a'_r,b'_r)$ 
satisfy 
$$\sum_{i=1}^r (a_i,b_i)\equiv\sum_{i=1}^r (a'_i,b'_i)\bmod \Z^2$$
then the product 
$$
F(z)=\prod_{i=1}^r \frac{\h_{a_i,b_i}(z)}{\h_{a'_i,b'_i}(z)}
$$
becomes an elliptic function with respect to the lattice 
$L_\tau=\Z\tau+\Z$, i.e., it is meromorphic on $\C$ and satisfies
$$F(z)=F(z+1)=F(z+\tau).$$

\begin{fact}[Jacobi's derivative formula]
\label{fact:derivative}
$$\frac{\pa}{\pa z}\h_{\frac{1}{2},\frac{1}{2}}(z,\tau)\Big|_{z=0}
=-\pi\h_{0,0}(0,\tau)\h_{0,\frac{1}{2}}(0,\tau)
\h_{\frac{1}{2},0}(0,\tau).
$$
\end{fact}

\begin{fact}[Transformation formulas]
\label{fact:inversion}
\begin{align*}
\h_{a,b}(z,\tau+1)
&=\ex\big(\frac{a(1-a)}{2}\big)\h_{a,a+b-\frac{1}{2}}(z,\tau),\\
\h_{a,b}\big(\frac{z}{\tau},\frac{-1}{\tau}\big)
&=\ex(ab)\sqrt{\frac{\tau}{\ii}}
\ex\big(\frac{z^2}{2\tau}\big)\h_{b,-a}(z,\tau),
\end{align*}
where $\sqrt{{\tau}/{\ii}}$ is positive when $\tau$ is purely imaginary.
\end{fact}

\begin{fact}[Addition formulas, Jacobi's identity]
\label{fact:addition}
\begin{align*}
\h_{0,0}(z_1+z_2)\h_{0,0}(z_1-z_2)\h_{0,0}(0)^2
&= \h_{0,0}(z_1)^2\h_{0,0}(z_2)^2
+\h_{\frac{1}{2},\frac{1}{2}}(z_1)^2\h_{\frac{1}{2},\frac{1}{2}}(z_2)^2\\
&= \h_{0,\frac{1}{2}}(z_1)^2\h_{0,\frac{1}{2}}(z_2)^2
+\h_{\frac{1}{2},0}(z_1)^2\h_{\frac{1}{2},0}(z_2)^2,\\
\h_{0,\frac{1}{2}}(z_1+z_2)\h_{0,\frac{1}{2}}(z_1-z_2)\h_{0,\frac{1}{2}}(0)^2
&= \h_{0,0}(z_1)^2\h_{0,0}(z_2)^2
-\h_{\frac{1}{2},0}(z_1)^2\h_{\frac{1}{2},0}(z_2)^2\\
&= \h_{0,\frac{1}{2}}(z_1)^2\h_{0,\frac{1}{2}}(z_2)^2
-\h_{\frac{1}{2},\frac{1}{2}}(z_1)^2\h_{\frac{1}{2},\frac{1}{2}}(z_2)^2,\\
\h_{\frac{1}{2},0}(z_1+z_2)\h_{\frac{1}{2},0}(z_1-z_2)\h_{\frac{1}{2},0}(0)^2
&= \h_{0,0}(z_1)^2\h_{0,0}(z_2)^2
-\h_{0,\frac{1}{2}}(z_1)^2\h_{0,\frac{1}{2}}(z_2)^2\\
&= \h_{\frac{1}{2},0}(z_1)^2\h_{\frac{1}{2},0}(z_2)^2
-\h_{\frac{1}{2},\frac{1}{2}}(z_1)^2\h_{\frac{1}{2},\frac{1}{2}}(z_2)^2,\\
\h_{\frac{1}{2},\frac{1}{2}}(z_1+z_2)\h_{\frac{1}{2},\frac{1}{2}}
(z_1-z_2)\h_{0,0}(0)^2
&= \h_{\frac{1}{2},\frac{1}{2}}(z_1)^2\h_{0,0}(z_2)^2
-\h_{0,0}(z_1)^2\h_{\frac{1}{2},\frac{1}{2}}(z_2)^2\\
&= \h_{0,\frac{1}{2}}(z_1)^2\h_{\frac{1}{2},0}(z_2)^2
-\h_{\frac{1}{2},0}(z_1)^2\h_{0,\frac{1}{2}}(z_2)^2,\\
\h_{0,0}(0)^4&=\h_{0,\frac{1}{2}}(0)^4+\h_{\frac{1}{2},0}(0)^4.
\end{align*}
\end{fact}

\subsection{Formulas for $\tau=\ii$}
In this subsection, we obtain several formulas for $\h_{a,b}(z,\ii)$ 
in the case of $\tau=\ii$.

\begin{lemma} We have
\label{lem:i-times}
$$\h_{a,b}(\ii z,\ii)=\ex(ab)\exp(\pi z^2)\h_{-b,a}(z,\ii),$$
$$\begin{array}{ll}
\h_{0,0}(\ii z,\ii)=\exp(\pi z^2)\h_{0,0}(z,\ii),&
\h_{0,\frac{1}{2}}(\ii z,\ii)=\exp(\pi z^2)\h_{\frac{1}{2},0}(z,\ii),\\[2mm]
\h_{\frac{1}{2},0}(\ii z,\ii)=\exp(\pi z^2)\h_{0,\frac{1}{2}}(z,\ii),&
\h_{\frac{1}{2},\frac{1}{2}}(\ii z,\ii)=\ii\exp(\pi z^2)
\h_{\frac{1}{2},\frac{1}{2}}(\ii z,\ii),
\end{array}
$$
$$\h_{0,\frac{1}{2}}(0,\ii)=\h_{\frac{1}{2},0}(0,\ii)=
\frac{\h_{0,0}(0,\ii)}{\sqrt[4]{2}}.
$$
\end{lemma}

\begin{proof}
For the $\ii$-multiple formulas, 
we have only to substitute $\tau=\ii$ into the second formula 
for $\h_{-a,-b}$ in Fact \ref{fact:inversion}. 
We have $\h_{0,\frac{1}{2}}(0)=\h_{\frac{1}{2},0}(0)$ by 
substituting $z=0$ into the identity between 
$\h_{0,\frac{1}{2}}(\ii z)$ and $\h_{\frac{1}{2},0}(z)$. 
By Jacobi's identity, we have 
$\h_{0,0}(0)^4=2\h_{0,\frac{1}{2}}(0)^4$. 
Note that $\h_{0,0}(0)$ and $\h_{0,\frac{1}{2}}(0)$ take 
positive real values.
\end{proof}

\begin{lemma}
\label{lem:(i+1)-times}
We have
\begin{align*}
\h_{0,0}((1+\ii)z,\ii)&=
\frac{\h_{0,0}(0,\ii)\h_{0,\frac{1}{2}}(z,\ii)\h_{\frac{1}{2},0}(z,\ii)}
{\exp(\pi\ii(1+\ii) z^2)
\h_{0,\frac{1}{2}}(0,\ii)\h_{\frac{1}{2},0}(0,\ii)},\\
\h_{\frac{1}{2},\frac{1}{2}}((1+\ii)z,\ii)&=\ex\big(\frac{1}{8}\big)
\frac{\h_{0,0}(0,\ii)\h_{0,0}(z,\ii)\h_{\frac{1}{2},\frac{1}{2}}(z,\ii)}
{\exp(\pi\ii(1+\ii) z^2)
\h_{0,\frac{1}{2}}(0,\ii)\h_{\frac{1}{2},0}(0,\ii)},\\
\h_{0,\frac{1}{2}}((1+\ii)z,\ii)\h_{\frac{1}{2},0}((1+\ii)z,\ii)
&=
\frac{\h_{0,0}(z,\ii)^4-\h_{0,\frac{1}{2}}(z,\ii)^2\h_{\frac{1}{2},0}(z,\ii)^2}
{\exp(2\pi\ii(1+\ii) z^2)
\h_{0,\frac{1}{2}}(0,\ii)\h_{\frac{1}{2},0}(0,\ii)}.
%
\end{align*}
\end{lemma}
\begin{proof}
We set 
$$\eta(z)=\exp(\pi\ii(1+\ii)z^2)\h_{0,0}((1+\ii)z,\ii).$$
Since $\h_{0,0}(z)$ has simple zero at $z=\dfrac{\ii+1}{2}$, 
the function $\eta(z)$ has simple zero at $z=\dfrac{1}{2},\dfrac{\ii}{2}$.
By using the quasi periodicity of $\h_{0,0}(z)$, we can show 
that 
$$\eta(z+1)=-\eta(z),\quad \eta(z+\ii)=-\exp(-2\pi\ii(\ii+2z))\eta(z).$$
Thus the function 
$$\frac{\eta(z)}{\h_{0,\frac{1}{2}}(z)\h_{\frac{1}{2},0}(z)}$$
is a holomorphic elliptic function with respect to $L_\ii$; it is a constant.
We can determine this constant by putting $z=0$. 
The second formula is obtained by the substitution $z+\dfrac{1}{2}$ 
into $z$ for the first formula.
We show the third formula. 
By Fact \ref{fact:addition} 
for $z_1=z$ and $z_2=\ii z$, we have 
$$\h_{0,\frac{1}{2}}(z+\ii z)\h_{0,\frac{1}{2}}(z-\ii z)\h_{0,\frac{1}{2}}(0)^2
=\h_{0,0}(z)^2\h_{0,0}(\ii z)^2
-\h_{\frac{1}{2},0}(z)^2\h_{\frac{1}{2},0}(\ii z)^2.$$
This identity together with Lemma \ref{lem:i-times} leads the third formula. 
\end{proof}

\subsection{Formulas for $\tau=\z$} 
In this subsection, we obtain several formulas for $\h_{a,b}(z,\z)$ 
in the case of $\tau=\z=\dfrac{1+\sqrt{3}\ii}{2}$.

\begin{lemma}
\label{lem:omega-times}
We have
\begin{align*}
\h_{a,b}(\w z,\z)&=\ex\big(\frac{a^2}{2}+ab-\frac{1}{24}\big)
\ex\big(\dfrac{z^2}{2\z}\big)\h_{-a-b-\frac{1}{2},a}(z,\z),\\
\h_{a,b}(\w^2 z,\z)&=\ex\big(ab+\frac{b^2+b}{2}+\frac{1}{24}\big)
\ex\big(\dfrac{z^2}{2\w}\big)\h_{b,-a-b-\frac{1}{2}}(z,\z),
\end{align*}
$$\begin{array}{ll}
\h_{0,0}(\w z,\z)=\ex\big(\dfrac{-1}{24}\big)\ex\big(\dfrac{z^2}{2\z}\big)
\h_{\frac{1}{2},0}(z,\z),&  
\h_{0,0}(\w^2 z,\z)=\ex\big(\dfrac{1}{24}\big)\ex\big(\dfrac{z^2}{2\w}\big)
\h_{0,\frac{1}{2}}(z,\z),\\
\h_{0,\frac{1}{2}}(\w z,\z)=\ex\big(\dfrac{-1}{24}\big)
\ex\big(\dfrac{z^2}{2\z}\big)\h_{0,0}(z,\z),&
\h_{0,\frac{1}{2}}(\w^2 z,\z)=\ex\big(\dfrac{-1}{12}\big)
\ex\big(\dfrac{z^2}{2\w}\big)\h_{\frac{1}{2},0}(z,\z),\\
\h_{\frac{1}{2},0}(\w z,\z)=\ex\big(\dfrac{1}{12}\big)
\ex\big(\dfrac{z^2}{2\z}\big)\h_{0,\frac{1}{2}}(z,\z),&
\h_{\frac{1}{2},0}(\w^2 z,\z)=\ex\big(\dfrac{1}{24}\big)
\ex\big(\dfrac{z^2}{2\w}\big)\h_{0,0}(z,\z),\\
\h_{\frac{1}{2},\frac{1}{2}}(\w z,\z)=\w\ex\big(\dfrac{z^2}{2\z}\big)
\h_{\frac{1}{2},\frac{1}{2}}(z,\z),&
\h_{\frac{1}{2},\frac{1}{2}}(\w^2 z,\z)=\w^2\ex\big(\dfrac{z^2}{2\w}\big)
\h_{\frac{1}{2},\frac{1}{2}}(z,\z),\\
\end{array}
$$
where $\w=\z^2=\dfrac{-1+\sqrt{3}\ii}{2}$.
\end{lemma}
\begin{proof}
Fact \ref{fact:inversion} yields that 
$$\begin{array}{l}
 \hspace{4mm} \h_{a,b}\big(\dfrac{z}{\z},\dfrac{-1}{\z}\big)
=\ex(ab)\ex(\dfrac{-1}{24}\big)\ex\big(\dfrac{z^2}{2\z}\big)
\h_{b,-a}(z,\z),\\[3mm]
=\h_{a,b}(-\w z,\z-1)
=\ex\big(\frac{a(a-1)}{2}\big)\h_{a,-a+b+\frac{1}{2}}(-\w z,\z)=
\ex\big(\frac{a(a-1)}{2}\big)\h_{-a,a-b-\frac{1}{2}}(\w z,\z).
\end{array}
$$
By rewriting $(a',b')=(-a,a-b-\frac{1}{2})$ i.e., 
$(a,b)=(-a',-a'-b'-\frac{1}{2})$ for the identity 
$$
\ex(ab)\ex(\dfrac{-1}{24}\big)\ex\big(\dfrac{z^2}{2\z}\big)
\h_{b,-a}(z,\z)=
\ex\big(\frac{a(a-1)}{2}\big)\h_{-a,a-b-\frac{1}{2}}(\w z,\z),$$
we have the first formula. To get the second formula, 
substitute $z=\w^2 z$ into the first formula. 
These formulas yield the others. 
%
\end{proof}

\begin{lemma}
\label{lem:(z+1)-times}
For $\tau=\z$, we have 
\begin{align*}
\h_{0,0}((1+\z)z)&=
\frac{\ex(\frac{1}{8})\ex((\w^2+\frac{\w}{2})z^2)}{\h_{0,0}(0)^2}
\h_{\frac{1}{2},0}(z)\{\h_{0,0}(z)^2-\ii \h_{0,\frac{1}{2}}(z)^2\},\\
\h_{0,\frac{1}{2}}((1+\z)z)&=
\frac{\ex(\frac{1}{8})\ex((\w^2+\frac{\w}{2})z^2)}{\h_{0,\frac{1}{2}}(0)^2}
\h_{0,0}(z)\{\h_{0,\frac{1}{2}}(z)^2-\h_{\frac{1}{2},0}(z)^2\},\\
\h_{\frac{1}{2},0}((1+\z)z)&=
\frac{\ex((\w^2+\frac{\w}{2})z^2)}{\h_{\frac{1}{2},0}(0)^2}
\h_{0,\frac{1}{2}}(z)\{\h_{0,0}(z)^2+\ii\h_{\frac{1}{2},0}(z)^2\},\\
\h_{\frac{1}{2},\frac{1}{2}}((1+\z)z)&=
\frac{\ex((\w^2+\frac{\w}{2})z^2)}{\h_{0,0}(0)^2}
\h_{\frac{1}{2},\frac{1}{2}}(z)\{\h_{0,0}(z)^2+\ii\h_{0,\frac{1}{2}}(z)^2\}.\\
\end{align*}
\end{lemma}
\begin{proof}
We apply addition formulas in Fact \ref{fact:addition} to 
$z_1=z$ and $z_2=\z z$, and use Lemma \ref{lem:omega-times}.
For example, we have
$$\h_{0,0}((1+\z)z)\h_{0,0}((1-\z)z)\h_{0,0}(0)^2
= \h_{0,\frac{1}{2}}(z)^2\h_{0,\frac{1}{2}}(\z z)^2
+\h_{\frac{1}{2},0}(z)^2\h_{\frac{1}{2},0}(\z z)^2,
$$
$$
\h_{0,0}((1-\z)z)=\h_{0,0}(-\w z)=\h_{0,0}(\w z)
=\ex\big(\frac{-1}{24}\big) \ex\big(\frac{z^2}{2\z}\big)
\h_{\frac{1}{2},0}(z),
$$
$$
\h_{0,\frac{1}{2}}(\z z)^2=
\h_{0,\frac{1}{2}}(-\w^2z)^2=\h_{0,\frac{1}{2}}(\w^2z)^2=
\ex\big(\frac{-1}{6}\big) \ex\big(\frac{z^2}{\w}\big)
\h_{\frac{1}{2},0}(z)^2,
$$
$$
\h_{\frac{1}{2},0}(\z z)^2=
\h_{\frac{1}{2},0}(-\w^2z)^2=\h_{\frac{1}{2},0}(\w^2z)^2=
\ex\big(\frac{1}{12}\big) \ex\big(\frac{z^2}{\w}\big)
\h_{0,0}(z)^2,
$$
which yield the first formula. 
\end{proof}

\begin{lemma}
\label{lem:hi-theta}
Some theta constants $\h_{a,b}(0,\z)$ are related as follows:
$$
\h_{0,\frac{1}{2}}(0,\z)=\ex\big(\frac{-1}{24}\big)\h_{0,0}(0,\z),\quad 
\h_{\frac{1}{2},0}(0,\z)=\ex\big(\frac{1}{24}\big)\h_{0,0}(0,\z),
$$
$$
\h_{\frac{5}{6},\frac{1}{3}}(0,\z)=
\ex\big(\frac{-1}{8}\big)\h_{\frac{1}{3},\frac{1}{3}}(0,\z),\quad 
\h_{\frac{1}{3},\frac{5}{6}}(0,\z)=
\ex\big(\frac{-17}{24}\big)\h_{\frac{1}{3},\frac{1}{3}}(0,\z).
$$
$$
\h_{\frac{1}{3},\frac{1}{3}}(0,\z)
=\ex\big(\frac{1}{18}\big)\frac{1}{\sqrt[3]{2}}\h_{0,0}(0,\z),\quad 
\h_{\frac{1}{6},\frac{1}{6}}(0,\z)
=\ex\big(\frac{1}{72}\big)\frac{\sqrt[4]{3}}{\sqrt[3]{2}}\h_{0,0}(0,\z).
$$
\end{lemma}

\begin{proof}
By substituting $z=0$ and $z=(\z+1)/3$ into formulas 
in Lemma \ref{lem:omega-times}, we have 
the formulas in the first and second lines in this lemma.  
We show the formulas in the third line. 
Substitute $z=(\zeta+1)/3$ and $z=(\zeta+1)/6$  into the first formula in 
Lemma \ref{lem:(z+1)-times}.
Then we have 
\begin{align*}
\h_{0,0}(\z)
&=\frac{\ex(\frac{1}{8})\ex\big((\w^2+\frac{\w}{2})\frac{(\z+1)^2}{9}\big)}
{\h_{0,0}(0)^2}
\h_{\frac{1}{2},0}\big(\frac{\z+1}{3}\big)
\Big\{\h_{0,0}\big(\frac{\z+1}{3}\big)^2
-\ii\h_{0,\frac{1}{2}}\big(\frac{\z+1}{3}\big)^2\Big\},
\\
\h_{0,0}\big(\frac{\z}{2}\big)
&=\frac{\ex(\frac{1}{8})\ex\big((\w^2+\frac{\w}{2})\frac{(\z+1)^2}{36}\big)}
{\h_{0,0}(0)^2}
\h_{\frac{1}{2},0}\big(\frac{\z+1}{6}\big)
\Big\{\h_{0,0}\big(\frac{\z+1}{6}\big)^2
-\ii\h_{0,\frac{1}{2}}\big(\frac{\z+1}{6}\big)^2\Big\}.
\end{align*}
By using shown formulas in this lemma, 
we can transform these identities into 
\begin{align*}
\h_{0,0}(0,\z)^3&=\frac{2}{\z}\h_{\frac{1}{3},\frac{1}{3}}(0,\z)^3,\\
\h_{0,0}(0,\z)^3&=\h_{\frac{1}{3},\frac{1}{3}}(0,\z)
\big(\h_{\frac{1}{6},\frac{1}{6}}(0,\z)^2-\z
\h_{\frac{1}{3},\frac{1}{3}}(0,\z)^2\big).
\end{align*}
Note that the last identity is equivalent 
$$
\h_{\frac{1}{6},\frac{1}{6}}(0,\z)^2
=\frac{\h_{0,0}(0,\z)^3+\z\h_{\frac{1}{3},\frac{1}{3}}(0,\z)^3}
{\h_{\frac{1}{3},\frac{1}{3}}(0,\z)}
=\frac{\z+1}{2}\cdot 
\frac{\h_{0,0}(0,\z)^3}{\h_{\frac{1}{3},\frac{1}{3}}(0,\z)}.
$$
By numerical computations, we can see that the identity 
$$\h_{\frac{1}{3},\frac{1}{3}}(0,\z)
=\ex\big(\frac{1}{18}\big)\frac{1}{\sqrt[3]{2}}\h_{0,0}(0,\z)$$
holds. This identity yields that 
$$ 
\h_{\frac{1}{6},\frac{1}{6}}(0,\z)^2=
\ex(\frac{1}{36})
\frac{\sqrt{3}}{\sqrt[3]{4}}
\h_{0,0}(0,\z)^2.
$$
By numerical computations,  
we can select a square root of $\ex(1/36)$ 
so that identity between 
$\h_{\frac{1}{6},\frac{1}{6}}(0,\z)$ and $\h_{0,0}(0,\z)$ holds.
\end{proof}

\section{The Schwarz map for $(\a,\b,\g)=(1/4,0,1/2)$}
We study the Schwarz map for $(\a,\b,\g)=(1/4,0,1/2)$ and 
its inverse by using an elliptic curve with $\ii$-action 
and $\h_{a,b}(z,\ii)$.

\subsection{Abel-Jacobi map for $C_\ii$}
Let $C_\ii$ be an algebraic curve in $\P^2$ defined by 
$$C_\ii:s_2^4=s_0s_1^2(s_1-s_0).$$
By affine coordinates $(t,u)=(s_1/s_0,s_2/s_0)$, 
$C_\ii$ is expressed by 
$$u^4=t^2(t-1).$$
Note that the point $(t,u)=(0,0)$ in $C_\ii$ is a node. 
We use the same symbol $C_\ii$ for a non-singular model of $C_\ii$. 
By a projection $\pr$ from the non-singular model $C_\ii$ to 
the complex projective line $\P^1$ arising from 
$$C_\ii\ni (t,u)\mapsto t\in \C,$$ 
we regard $C_\ii$ as a branched covering $\P^1$ 
with a covering transformation $\rho_\ii$ arising from a map
$$\rho_\ii:C_\ii\ni (t,u)\mapsto (t,\ii u)\in C_\ii.$$
The branch points of $\pr$ are $t=0,1,\infty$.
Each preimage of $pr^{-1}(1)$ and $pr^{-1}(\infty)$ consists of a point;  
$P_1=pr^{-1}(1)$ and $P_\infty= pr^{-1}(\infty)$ are expressed as
$(t,u)=(1,0)$ and $[s_0,s_1,s_2]=[0,1,0]$, respectively.
On the other hand, the preimage $pr^{-1}(0)$ consists 
of two points, which are denoted by $P_{0,1}$ and $P_{0,2}$. 
The point $P_{0,1}$ corresponds to 
$$\lim_{\substack{x\to 0\\x\in (0,1)}} (x,\sqrt[4]{x^2(x-1)}),\quad 
\arg x^2(x-1)=\pi$$
for $x$ in the open interval $(0,1)$, and $P_{0,2}$ is given by 
$\rho_\ii(P_{0,1})$.
By the Hurwitz formula, $C_\ii$ is an elliptic curve.

Let $I_{1\infty}$ be an oriented path in $C_\ii$ given by 
$$(x,\sqrt[4]{x^2(x-1)})\in C_\ii,\quad x\in [1,\infty],$$
where $\sqrt[4]{x^2(x-1)}$ takes real values for $x\in [1,\infty)$ and
the interval $[1,\infty]$ is naturally oriented.
We define a cycle $B$ by $I_{1\infty}-\rho_\ii\cdot I_{1\infty}$ 
and a cycle $A$ by $\rho_\ii\cdot B$. Since 
$$B\cdot A=1,$$
$A$ and $B$ form a basis of $H_1(C_\ii,\Z)$.

The space of holomorphic $1$-forms on $C_\ii$ is one dimensional
and it is spanned by a form expressed by
$$\f=\frac{udt}{t(t-1)}=
\frac{dt}{\sqrt[4]{t^2(t-1)^3}}.$$
The period integral $\int_B\f$ is evaluated as
$$(1-\ii )\int_1^\infty \frac{dt}{\sqrt[4]{t^2(t-1)^3}}
=(1-\ii )B\big(\frac{1}{4},\frac{1}{4}\big).
$$
On the other hand, we have
$$\int_A\f=\int_{\rho_\ii(B)}\f=\int_B\rho_\ii^*(\f)=\ii \int_B\f.$$
We normalize $\f$ to $\f_1$ as
$$\f_1=\frac{1}{(1-\ii )B(\frac{1}{4},\frac{1}{4})}\f.$$
Then we have 
$$\int_B\f_1=1,\quad \int_A\f_1=\ii $$ and the Abel-Jacobi map
$$\jmath_\ii: C_\ii\ni P=(x,\sqrt[4]{x^2(x-1)}) \mapsto 
z=\int_{P_1}^P \f_1\in E_\ii=\C/L_\ii,$$
where $L_\ii=\Z\ii+\Z \subset \C$. 
The map $\jmath_\ii$ is an isomorphism between $C_\ii$ and $E_\ii$. 

\begin{proposition}
\label{prop:aj-map}
The Abel-Jacobi map $\jmath_\ii$ sends points 
$P_1$, $P_\infty$, $P_{0,1}$ and $P_{0,2}$ 
to 
$$\jmath_\ii(P_1)=0,\quad \jmath_\ii(P_\infty)=\frac{\ii+1}{2},\quad 
\jmath_\ii(P_{0,1})=\frac{\ii}{2},\quad \jmath_\ii(P_{0,2})=\frac{1}{2}
$$
as elements of $E_\ii$.
\end{proposition}
\begin{proof}
It is clear that $\jmath_\ii(P_1)=0$ and $\jmath_\ii(P_\infty)=\dfrac{\ii+1}{2}$.
We have 
\begin{align*}
\jmath_\ii(P_{0,1})
&=\frac{1}{(1-\ii)B(\frac{1}{4},\frac{1}{4})}
\int_1^0 \exp(\pi \ii/4) \frac{\sqrt[4]{s^2(1-s)}ds}{s(s-1)}
=\frac{\ii}{\sqrt{2}}\cdot \frac{\G(\frac{1}{2})^2\G(\frac{1}{4})}
{\G(\frac{1}{4})^2\G(\frac{3}{4})}\\
&=\frac{\ii}{\sqrt{2}}\cdot \frac{\pi}
{\pi/\sin\frac{\pi}{4}}=\frac{\ii}{2}.
\end{align*}
Since $P_{0,2}=\rho_\ii(P_{0,1})$, $\jmath_\ii(P_{0,2})$ is equal to 
$\ii \jmath_\ii(P_{0,1})=\dfrac{-1}{2}\equiv \dfrac{1}{2}$ $\bmod {L_\ii}$.
\end{proof}

We consider the relation between the Abel-Jacobi map $\jmath_\ii$ and  
the Schwarz map 
\begin{equation}
\label{eq:S(1/4,0,1/2)}
x\mapsto \frac{f_1(x)}{(1-\ii)f_2(x)}
=\frac{2\sqrt{2}\ii}{B\big(\frac{1}{4},\frac{1}{4}\big)}
\sqrt[4]{1-x}F(\frac{1}{4},\frac{1}{2},\frac{5}{4},1-x)
\end{equation}
for $\CF(\frac{1}{4},0,\frac{1}{2})$.
By Corollary \ref{cor:monod}, its monodromy group is generated by 
the three transformations 
$$N_0:z\mapsto -z+\ii,\quad N_1:z\mapsto \ii z,\quad 
(N_0N_1)^{-1}:z\mapsto \ii z+1,$$
and this group is isomorphic to the semi-direct product 
$\la \ii\ra \ltimes \Z[\ii]$. 
Note that the information of a branch of $\sqrt[4]{x^2(x-1)}$ is 
lost in the Schwarz map.  
Thus we can regard the Schwarz map as the Abel-Jacobi map $\jmath_\ii$ 
modulo the actions of $\rho_\ii$ and $\ii$; that is 
$$C_\ii/\la \rho_\ii\ra \ni x\mapsto \int_1^x \f_1 \in E_\ii/\la \ii\ra,$$ 
where $\la \rho_\ii\ra$ and $\la \ii\ra$ are the groups generated by 
$\rho_\ii$ and $\ii$, respectively.

\subsection{The inverse of $\jmath_\ii$}
In this subsection, we express the inverse of the Abel-Jacobi map $\jmath_\ii$ 
in terms of $\h_{a,b}(z,\tau)$. 
We fix the variable $\tau$ to $\ii $ and denote $\h_{a,b}(z,\ii )$ 
by $\h_{a,b}(z)$ in short. Since 
the pull-backs ${\jmath_\ii^{-1}}^*(t)$ and ${\jmath_\ii^{-1}}^*(u)$ are 
elliptic functions with respect to the lattice $L_\ii$, 
they can be expressed 
as $${\jmath_\ii^{-1}}^*(t)=\theta_t(z),\quad {\jmath_\ii^{-1}}^*(u)=\theta_u(z)$$
in terms of $\h_{a,b}(z)$.
It turns out that the map 
$$E_\ii\ni z \mapsto (\theta_t(z),\theta_u(z))\in C_\ii$$
is  the inverse of $\jmath_\ii$.

\begin{theorem}
\label{th:inverse}
The inverse of $\jmath_\ii:C_\ii\ni(t,u)\mapsto z\in E_\ii$ is given by
\begin{align*}
t&=2\frac{\h_{0,\frac{1}{2}}(z,\ii)^2\h_{\frac{1}{2},0}(z,\ii)^2}
{\h_{0,0}(z,\ii)^4}
=1-\frac{\h_{\frac{1}{2},\frac{1}{2}}(z,\ii)^4}{\h_{0,0}(z,\ii)^4},\\
u&=-(1-\ii)
\frac{\h_{0,\frac{1}{2}}(z,\ii)
\h_{\frac{1}{2},0}(z,\ii)\h_{\frac{1}{2},\frac{1}{2}}(z,\ii)}
{\h_{0,0}(z,\ii)^3}.\\
\end{align*} 
The holomorphic $1$-form $\f=\dfrac{udt}{t(t-1)}$ on $C_\ii$ corresponds to 
$$
2(1-\ii)\pi\h_{0,0}(0,\ii)^2dz=(1-\ii)B(\frac{1}{4},\frac{1}{4})dz
$$
by the Abel-Jacobi map $\jmath_\ii$.

\end{theorem}
\begin{proof}
We regard the coordinate $t$ of $C_\ii$ as a meromorphic function on $C_\ii$.
Its divisor is 
$$2P_{0,1}+2P_{0,2}-4P_\infty.$$
We construct an elliptic function for $L_\ii$ with zero of order $2$ at
$z=\dfrac{\ii }{2},\dfrac{1}{2}$ and pole of order $4$ at 
$z=\dfrac{\ii +1}{2}$.
Since 
$$2\cdot\big(0,\frac{1}{2}\big)+2\cdot\big(\frac{1}{2},0\big)\equiv 
4\cdot\big(\frac{1}{2},\frac{1}{2}\big)\bmod \Z^2,
$$
the function 
$$\frac{\h_{0,\frac{1}{2}}(z)^2\h_{\frac{1}{2},0}(z)^2}{\h_{0,0}(z)^4}$$
becomes an elliptic function for $L_\ii$. Moreover, it has zero 
of order $2$ at $z=\dfrac{\ii}{2},\dfrac{1 }{2}$, and 
pole of order $4$ at $z=\dfrac{\ii +1}{2}$, since 
$\h_{a,b}(z)=0$ if and only if 
$z\equiv \big(-a+\dfrac{1}{2}\big)\ii +\big(-b+\dfrac{1}{2}\big)$ $\mod \Z^2$.
Thus the pull-back $F(P)$ of this function under the map $\jmath_\ii$ is 
a constant multiple of $t$ by Proposition \ref{prop:aj-map}. 
Let us determine this constant. 
Lemma \ref{lem:i-times} yields that  
$$
\frac{\h_{0,\frac{1}{2}}(0)^2\h_{\frac{1}{2},0}(0)^2}{\h_{0,0}(0)^4}
=
\frac{\h_{0,\frac{1}{2}}(0)^4}{\h_{0,0}(0)^4}
=\frac{1}{2}.$$
Thus $2F(P)$ is equal to $t$.

Similarly we regard $t-1$ as a meromorphic function on $C_\ii$ whose
divisor is 
$$4P_1 -4P_\infty.$$
The function 
$$\frac{\h_{\frac{1}{2},\frac{1}{2}}(z)^4}{\h_{0,0}(z)^4}$$
becomes an elliptic function for $L_\ii$ with zero of order $4$ at $z=0$ 
and pole of order $4$ at $z=\dfrac{\ii+1}{2}$.
The pull-back of this function under the map $\jmath_\ii$ is 
a constant multiple of $t-1$. By substituting $P_{0,1}$ into this pull-back,
we can determine the constant. We have
$$t-1=-\frac{\h_{\frac{1}{2},\frac{1}{2}}(z)^4}{\h_{0,0}(z)^4}.
$$

By regarding the coordinate $u$ of $C_\ii$ as a meromorphic function on $C_\ii$, 
we see that its divisor is 
$$P_{0,1}+P_{0,2}+P_1-3P_\infty.$$
Thus it is the pull-back of 
$$
c\cdot \frac{\h_{0,\frac{1}{2}}(z)\h_{\frac{1}{2},0}(z)
\h_{\frac{1}{2},\frac{1}{2}}(z)}
{\h_{0,0}(z)^3}
$$
under $\jmath_\ii$, where $c$ is a constant.
Let us determine $c$. 
By $u^4=t^2(t-1)$, we have 
$$
c^4\cdot \frac{\h_{0,\frac{1}{2}}(z)^4\h_{\frac{1}{2},0}(z)^4
\h_{\frac{1}{2},\frac{1}{2}}(z)^4}
{\h_{0,0}(z)^{12}}
=\frac{4\h_{0,\frac{1}{2}}(z)^4\h_{\frac{1}{2},0}(z)^4}
{\h_{0,0}(z)^{8}}\cdot 
\frac{-\h_{\frac{1}{2},\frac{1}{2}}(z)^4}{\h_{0,0}(z)^{4}},
$$
which yields that $c^4=-4$, i.e., $c=\ii^k\cdot (1+\ii)$ for some 
$k\in \{0,1,2,3\}$.

By the expressions $t$, $t-1$ and $u$ 
in terms of $\h_{a,b}(z)$, it turns out that 
the holomorphic $1$-from $\f=\dfrac{udt}{t(t-1)}$ corresponds to 
\begin{align*}
& \ii^k(1+\ii)\cdot \frac{\h_{0,\frac{1}{2}}(z)\h_{\frac{1}{2},0}(z)
\h_{\frac{1}{2},\frac{1}{2}}(z)}{\h_{0,0}(z)^3}\cdot 
\frac{\h_{0,0}(z)^4}{2\h_{0,\frac{1}{2}}(z)^2\h_{\frac{1}{2},0}(z)^2}\cdot
\frac{-\h_{0,0}(z)^4}{\h_{\frac{1}{2},\frac{1}{2}}(z)^4}\\
& \cdot 
\frac{4\{\h_{0,0}(z)^3\h_{0,0}(z)'\h_{\frac{1}{2},\frac{1}{2}}(z)^4
-\h_{\frac{1}{2},\frac{1}{2}}(z)^3\h_{\frac{1}{2},\frac{1}{2}}(z)'
\h_{0,0}(z)^4\}}{\h_{0,0}(z)^8}dz\\
&=-2\ii^k(1+\ii)\cdot \frac{\{\h_{0,0}(z)'\h_{\frac{1}{2},\frac{1}{2}}(z)
-\h_{\frac{1}{2},\frac{1}{2}}(z)'\h_{0,0}(z)\}}
{\h_{0,\frac{1}{2}}(z)\h_{\frac{1}{2},0}(z)}dz,
\end{align*}
which should be a constant multiple of $dz$.
By putting $z=0$ and using Fact \ref{fact:derivative}, we have
$$\f=-2\ii^k(1+\ii)\pi\h_{0,0}(0)^2\jmath_\ii^*(dz).$$
Since $\h_{0,0}(0)^2$ and 
$$B\big(\frac{1}{4},\frac{1}{4}\big)=
\int_1^\infty \f=\int_{\jmath_\ii(P_1)}^{\jmath_\ii(P_\infty)}
-2\ii^k(1+\ii)\pi\h_{0,0}(0)^2dz
=-2\ii^k(1+\ii)\pi\h_{0,0}(0)^2\cdot \frac{1+\ii}{2}$$
are positive real, $k$ is equal to $1$. Hence we 
have the expressions of $u$ and $\f$.
\end{proof}

\begin{cor}
Let $z\in E_\ii$ be the image of $(t,u)\in C_\ii$ under 
the Abel-Jacobi map $\jmath_\ii$. Then we have 
$$
\ii\frac{u^2}{t}=\frac{\h_{\frac{1}{2},\frac{1}{2}}(z)^2}
{\h_{0,0}(z)^2},\quad 
1+\ii\frac{u^2}{t}=\sqrt{2}\frac{\h_{0,\frac{1}{2}}(z)^2}
{\h_{0,0}(z)^2},\quad 
1-\ii\frac{u^2}{t}=\sqrt{2}\frac{\h_{\frac{1}{2},0}(z)^2}
{\h_{0,0}(z)^2}.
$$
Moreover, $\h_{a,b}(z)$'s satisfy relations
$$
\sqrt{2}\h_{0,\frac{1}{2}}(z)^2=
\h_{0,0}(z)^2+\h_{\frac{1}{2},\frac{1}{2}}(z)^2,\quad
\sqrt{2}\h_{\frac{1}{2},0}(z)^2=
\h_{0,0}(z)^2-\h_{\frac{1}{2},\frac{1}{2}}(z)^2.
$$
\end{cor}

\begin{proof}
The first identity is a direct consequence of Theorem \ref{th:inverse}.
The right hand side of the second identity is an elliptic function with
respect to $L_\ii$. It has zero of order $2$ at $\jmath_\ii(P_{0,1})$ 
and pole of order $2$ at $\jmath_\ii(P_{\infty})$. Since $P_{0,1}$ corresponds 
to the limit as $t\to 0$ given by the branch of $u$ with 
$\arg(u)=\dfrac{\pi}{4}$ 
on the interval $(0,1)$, $\lim\limits_{t\to 0}\ii\dfrac{u^2}{t}=-1$. 
By comparing the zero and pole of both functions, 
$1+\ii\dfrac{u^2}{t}$ is a constant multiple of the pull-back of
$\dfrac{\h_{0,\frac{1}{2}}(z)^2}
{\h_{0,0}(z)^2}$ under $\jmath_\ii$. We can determine this constant 
by the substitution $z=0$.
The third identity is obtained by the action of $\rho_\ii$ on the second 
identity. By eliminating $\ii\dfrac{u^2}{t}$ from these identities,
we have the relations among $\h_{a,b}(z)$'s
\end{proof}

\begin{cor}
We have 
$$\h_{0,0}(0,\ii)=\frac{\G(\frac{1}{4})}{\sqrt[4]{4\pi^3}}
=\frac{\sqrt[4]{\pi}}{\G(\frac{3}{4})},\quad 
\h_{0,\frac{1}{2}}(0,\ii)=\h_{\frac{1}{2},0}(0,\ii)=
\frac{\G(\frac{1}{4})}{\sqrt[4]{(2\pi)^3}}
=\frac{\sqrt[4]{\pi}}{\sqrt[4]{2}\G(\frac{3}{4})}.
$$
\end{cor}

\begin{proof}
By Theorem \ref{th:inverse}, 
we have 
$$2\pi\h_{0,0}(0)^2=B(\frac{1}{4},\frac{1}{4})
=\frac{\G\big(\frac{1}{4}\big)^2}{\sqrt{\pi}}.$$
Note that $\h_{0,0}(0)$ and $\G\big(\frac{1}{4}\big)$ are positive.  
To show the rest, use the inversion formula for the $\G$-function 
and Lemma \ref{lem:i-times}.
\end{proof}

\begin{cor}
\label{cor:inv-Schwarz}
The inverse of the Schwarz map (\ref{eq:S(1/4,0,1/2)}) 
for $\CF(\frac{1}{4},0,\frac{1}{2})$ 
is given by
$$
x=2\frac{\h_{0,\frac{1}{2}}(z)^2\h_{\frac{1}{2},0}(z)^2}{\h_{0,0}(z)^4}
=1-\frac{\h_{\frac{1}{2},\frac{1}{2}}(z)^4}{\h_{0,0}(z)^4}.
$$
\end{cor}
\begin{proof}
It is clear by Theorem \ref{th:inverse}. We can check this map 
is invariant under the action of $\la \ii\ra$ by Lemma \ref{lem:i-times}.
\end{proof}

\begin{cor} For any point $z$ around $0$, we have 
\label{cor:HGF-Theta}
$$-\frac{2\sqrt{2\pi}}{\G\big(\frac{1}{4}\big)^2}\cdot 
\frac{\h_{\frac{1}{2},\frac{1}{2}}(z,\ii)}{\h_{0,0}(z,\ii)}\cdot 
F\big(\frac{1}{4},\frac{1}{2},\frac{5}{4};
\frac{\h_{\frac{1}{2},\frac{1}{2}}(z,\ii)^4}{\h_{0,0}(z,\ii)^4}\big)
=z.
$$
\end{cor}
\begin{proof}
By Corollary \ref{cor:inv-Schwarz}, we have 
$$\frac{2\sqrt{2\pi}}{\G\big(\frac{1}{4}\big)^2}\cdot 
\frac{\h_{\frac{1}{2},\frac{1}{2}}(z,\ii)}{\h_{0,0}(z,\ii)}\cdot 
F\big(\frac{1}{4},\frac{1}{2},\frac{5}{4};
\frac{\h_{\frac{1}{2},\frac{1}{2}}(z,\ii)^4}{\h_{0,0}(z,\ii)^4}\big)
\equiv z
$$
modulo the monodromy group of $\CF\big(\frac{1}{4},0,\frac{1}{2}\big)$.
Since the both sides of the above become $0$ for $z=0$, 
their difference is represented as the group $\la \ii\ra$.
Consider the limit of the both sides as $z\to \dfrac{\ii}{2}$ along 
the imaginary axis.
Use
$$ 
\frac{\h_{\frac{1}{2},\frac{1}{2}}(\frac{\ii}{2},\ii)}
{\h_{0,0}(\frac{\ii}{2},\ii)}
=\ex\big(\frac{1}{2}\cdot \frac{-1}{2}\big)\cdot
\frac{\h_{0,\frac{1}{2}}(0,\ii)}
{\h_{\frac{1}{2},0}(0,\ii)}=-\ii, 
$$
and the Gauss-Kummer formula
$$
F(\a,\b,\g;1)=\frac{\G(\g)\G(\g-\a-\b)}{\G(\g-\a)\G(\g-\b)}$$
for $\re(\g-\a-\b)>0$.
\end{proof}

\subsection{$(1+\ii )$-multiplication}

\begin{theorem}
\label{th:(1+i)-multiple}
Let $z\in E_\ii$ be the image of $(t,u)\in C_\ii$ under the Abel-Jacobi map 
$\jmath_\ii$. Then we have 
\begin{equation}
\label{eq:(1+i)-mul}
\jmath_\ii^{-1}((1+\ii)z)=
\left(\left(\frac{t-2}{t}\right)^2,(1+\ii)\frac{u(2-t)}{t^2}\right).
\end{equation}
\end{theorem}

\begin{proof}
We set 
$$(t',u')=\jmath_\ii^{-1}((1+\ii)z).$$ 
By Theorem \ref{th:inverse}, we have
\begin{align*}
t'&=2\frac{\h_{0,\frac{1}{2}}((1+\ii)z)^2\h_{\frac{1}{2},0}((1+\ii)z)^2}
{\h_{0,0}((1+\ii)z)^4},\\
u'&=-(1-\ii)
\frac{\h_{0,\frac{1}{2}}((1+\ii)z)\h_{\frac{1}{2},0}((1+\ii)z)\h_{\frac{1}{2},\frac{1}{2}}((1+\ii)z)}
{\h_{0,0}((1+\ii)z)^3}.
\end{align*}
We transform them as 
\begin{align*}
t'&=2\frac{\h_{0,\frac{1}{2}}^4(0)\h_{\frac{1}{2},0}^4(0)}{\h_{0,0}^4(0)
\h_{0,\frac{1}{2}}(z)^4\h_{\frac{1}{2},0}(z)^4}
\cdot
\frac{(\h_{0,0}(z)^4-\h_{0,\frac{1}{2}}(z)^2\h_{\frac{1}{2},0}(z)^2)^2}
{\h_{0,\frac{1}{2}}(0)^2\h_{\frac{1}{2},0}(0)^2}\\
&= 2\frac{\h_{0,\frac{1}{2}}(0)^2\h_{\frac{1}{2},0}(0)^2}
{\h_{0,0}(0)^4}\cdot
\frac{(\h_{0,0}(z)^4-\h_{0,\frac{1}{2}}(z)^2\h_{\frac{1}{2},0}(z)^2)^2}
{\h_{0,\frac{1}{2}}(z)^4\h_{\frac{1}{2},0}(z)^4}=\left(\frac{2}{t}-1\right)^2,
\end{align*}
\begin{align*}
u'&=-(1-\ii)\cdot
\frac{\ex\big(\frac{1}{8}\big)
\h_{0,0}(0)\h_{0,0}(z)\h_{\frac{1}{2},\frac{1}{2}}(z)}
{\h_{0,\frac{1}{2}}(0)\h_{\frac{1}{2},0}(0)}
\cdot 
\frac{\h_{0,0}(z)^4-\h_{0,\frac{1}{2}}(z)^2\h_{\frac{1}{2},0}(z)^2}
{\h_{0,\frac{1}{2}}(0)\h_{\frac{1}{2},0}(0)}\\
& \cdot 
\frac{\h_{0,\frac{1}{2}}(0)^3\h_{\frac{1}{2},0}(0)^3}{\h_{0,0}(0)^3}\frac{1}
{\h_{0,\frac{1}{2}}(z)^3\h_{\frac{1}{2},0}(z)^3}\\
&=-\sqrt{2}\cdot \frac{\h_{0,\frac{1}{2}}(0)\h_{\frac{1}{2},0}(0)}
{\h_{0,0}(0)^2}\cdot 
\frac{\h_{0,0}(z)\h_{\frac{1}{2},\frac{1}{2}}(z)(\h_{0,0}(z)^4-\h_{0,\frac{1}{2}}(z)^2\h_{\frac{1}{2},0}(z)^2)}
{\h_{0,\frac{1}{2}}(z)^3\h_{\frac{1}{2},0}(z)^3}\\
&=-\frac{\h_{0,0}(z)^8\h_{\frac{1}{2},\frac{1}{2}}(z)\h_{0,\frac{1}{2}}(z)\h_{\frac{1}{2},0}(z)  }
{\h_{0,\frac{1}{2}}(z)^4\h_{\frac{1}{2},0}(z)^4\h_{0,0}(z)^3}+
\frac{\h_{0,0}(z)^4\h_{\frac{1}{2},\frac{1}{2}}(z)\h_{0,\frac{1}{2}}(z)\h_{\frac{1}{2},0}(z)}
{\h_{0,\frac{1}{2}}(z)^2\h_{\frac{1}{2},0}(z)^2\h_{0,0}(z)^3}
\\
&=\frac{4}{(1-\ii)}\frac{u}{t^2}-\frac{2}{(1-\ii)}\frac{u}{t}
=(1+\ii)\frac{u(2-t)}{t^2},
\end{align*}
by Lemma \ref{lem:(i+1)-times} and Theorem \ref{th:inverse}. 
\end{proof}

\section{The Schwarz map for $(\a,\b,\g)=(1/3,0,1/2)$}
In this section, we study the Schwarz map for $(\a,\b,\g)=(1/3,0,1/2)$ and 
its inverse by using an elliptic curve with $\z$-action and $\h_{a,b}(z,\z)$,
where $\z=\dfrac{1+\sqrt{3}\ii}{2}$.
\subsection{The Abel-Jacobi map for $C_\z$}
Let $C_\z$ be an algebraic curve in $\P^2$ defined by 
$$
C_\z: s_2^6=s_0^2s_1^3(s_1-s_0).
$$
By affine coordinates $(t,u)=(s_1/s_0,s_2/s_0)$, $C_\z$  is expressed as
$$u^6=t^3(t-1).$$
Note that $(t,u)=(0,0)$ is a triple node and 
$[s_0,s_1,s_2]=[0,1,0]$ is a node. We use the same symbol $C_\z$ for 
a non-singular model of $C_\z$.
We regard $C_\z$ as a cyclic $6$-fold covering of the $t$-space 
with covering transformation 
$$\rho_\z:(t,u)\mapsto (t,\z u),\quad \z=\frac{1+\sqrt{-3}}{2}.$$
The branching information of this covering is as in Table \ref{tab:6-covering}.
\begin{table}[htb]
\begin{center}
\begin{tabular}[htb]{|c|cccccc|}
\hline
ramification point& 
$P_{0,1}$&$P_{0,2}$ & $P_{0,3}$ & $P_1=(1,0)$ & 
$P_{\infty,1}$ &$P_{\infty,2}$\\
projected point&
$0$ &$0$ &$0$ &$1$ &$\infty$ &$\infty$ \\
ramification index& 
$2$ &$2$ &$2$ &$6$ &$3$ &$3$ \\
\hline
\end{tabular}
\end{center} 
\caption{Branching information} 
\label{tab:6-covering}
\end{table}
Here we set some points in the non-singular model $C_\z$ as follows:
$$
P_{0,1}=\lim_{\substack{t\to 0\\t\in (0,1)}}
(t,t^{1/2}(t-1)^{1/6}),\quad P_{0,2}=\rho_\z(P_{0,1}),\quad 
P_{0,3}=\rho_\z^2(P_{0,1}),
$$
$$
P_{\infty,1}=\lim_{\substack{t\to \infty\\t\in (1,\infty)}} 
(t,t^{1/2}(t-1)^{1/6}),\quad P_{\infty,2}=\rho_\z(P_{\infty,1}),$$
where $\arg(t)=\arg(t-1)=0$ on the open interval $I_\infty=(1,\infty)$ and 
$\arg(t)=0$, $\arg(t-1)=\pi$ on the open interval $I_0=(0,1)$.
By the Hurwitz formula, $C_\z$ is an elliptic curve.

We can regard $t$ and $u$ as meromorphic functions on $C_\z$. We give some 
meromorphic functions on $C_\z$ and their zero and pole divisors 
as in Table \ref{tab:divisor}.
\begin{table}[htb]
  \centering
  \begin{tabular}[htb]{c|c|c}
functions & zero divisor & pole divisor \\[2mm]
\hline
$t$ & $2P_{0,1}+2P_{0,2}+2P_{0,3}$& $3P_{\infty,1}+3P_{\infty,2}$\\[2mm]
$t-1$ & $6P_{1}$& $3P_{\infty,1}+3P_{\infty,2}$\\[2mm]
$1+\dfrac{1}{t-1}$ &$2P_{0,1}+2P_{0,2}+2P_{0,3}$&  $6P_{1}$\\[2mm]
$u$ & $P_{0,1}+P_{0,2}+P_{0,3}+P_1$ &$2P_{\infty,1}+2P_{\infty,2}$\\[2mm]
$\dfrac{u^2}{t}(=\sqrt[3]{t-1})$ & $2P_1$ &$P_{\infty,1}+P_{\infty,2}$\\[2mm]
$\dfrac{u^3}{t}\Big(=\sqrt{t(t-1)}\Big)$ 
& $P_{0,1}+P_{0,2}+P_{0,3}+3P_{1}$&$3P_{\infty,1}+3P_{\infty,2}$\\[2mm]
$\dfrac{u^3}{t(t-1)}\Big(=\sqrt{\dfrac{t}{t-1}}\Big)$ 
& $P_{0,1}+P_{0,2}+P_{0,3}$&$3P_{1}$\\[2mm]
$1+\dfrac{\z^4 t}{u^2}\Big(=1+\dfrac{\z^4}{\sqrt[3]{t-1}}\Big)$ 
& $2P_{0,1}$&$2P_{1}$\\[2mm]
$1+\dfrac{t}{u^2}\Big(=1+\dfrac{1}{\sqrt[3]{t-1}}\Big)$ 
& $2P_{0,2}$&$2P_{1}$\\[2mm]
$1+\dfrac{\z^2t}{u^2}\Big(=1+\dfrac{\z^2}{\sqrt[3]{t-1}}\Big)$ 
& $2P_{0,3}$&$2P_{1}$\\[2mm]
  \end{tabular}
  \caption{Meromorphic functions on $C_\z$}
  \label{tab:divisor}
\end{table}
Pay your attention to the last three meromorphic functions 
for the setting of branch of $u$.
Note that 
$$\left(1+\frac{t}{u^2}\right)\left(1+\frac{\z^2t}{u^2}\right)
\left(1+\frac{\z^4t}{u^2}\right)=1+\frac{1}{t-1}.$$
The preimage of $I_\infty$ under the natural projection consists of 
six copies $\rho_\z^i\cdot I_{\infty}$ $(i=0,1,\dots,5)$.  
Since the terminal points of $\rho_\z^2\cdot I_\infty$ 
coincide with that of $I_\infty$, 
the formal difference $B=\rho_\z^0\cdot I_\infty -\rho_\z^2\cdot I_\infty
=(1-\rho_\z^2)\cdot I_{\infty}$ is a cycle of $C_\z$. 
Let $A$ be the cycle $\rho_\z\cdot B$. Then the 
intersection number $B\cdot A$ of the cycles $B$ and $A$ is $1$. 
Thus the cycles $A$ and $B$ form a basis of the first homology group 
$H_1(C_\z,\Z)$ of $C_\z$.

A non-zero holomorphic $1$-form $\psi$ on $C_\z$ is given by 
$$\psi=\frac{t^2dt}{u^5}=\frac{udt}{t(t-1)}
=\frac{t^{1/2}(t-1)^{1/6}dt}{t(t-1)}.
$$
It is easy to see that 
$$\rho_\z^*(\psi)=\z \psi.$$
Note that 
$$\int_{I_\infty}\psi =\int_1^\infty t^{1/2-1}(t-1)^{1/6-1}dt 
=\int_0^1 s^{1/3-1}(1-s)^{1/6-1}ds
=B\Big(\frac{1}{3},\frac{1}{6}\Big),
$$ 
$$
\int_{A}\psi =\z(1-\z^2)B\Big(\frac{1}{3},\frac{1}{6}\Big),
\quad \int_{B}\psi =(1-\z^2)B\Big(\frac{1}{3},\frac{1}{6}\Big).
$$
We normalize $\psi$ to $\psi_1$ as
$$\psi_1=\frac{1}{(1-\z^2)B\Big(\frac{1}{3},\frac{1}{6}\Big)}\psi,$$
then we have 
$$\int_A\psi_1=\z,\quad \int_B\psi_1=1.$$
The Abel-Jacobi map is defined by 
$$
\jmath_\z:C_\z\ni P\mapsto  \int_{P_1}^P \psi_1 \in E_\z=\C/L_\z,
$$
where $L_\z=\Z\z+\Z\subset \C$. 
The map $\jmath_\z$ is an isomorphism between $C_\z$ and $E_\z$. 

\begin{proposition}
\label{prop:jmath-image}
We have 
$$\jmath_\z(P_1)=0,\quad\jmath_\z(P_{\infty,1})=\frac{\z+1}{3},\quad
\jmath_\z(P_{\infty,2})=\frac{2\z+2}{3},$$
$$
\jmath_\z(P_{0,1})=\frac{\z}{2},\quad
\jmath_\z(P_{0,2})=\frac{\z+1}{2},\quad
\jmath_\z(P_{0,3})=\frac{1}{2}$$
as elements of $E_\z$.
\end{proposition}

\begin{proof}
It is obvious that $\jmath_\z(P_1)=0$. It is easy to see that 
\begin{align*}
\jmath_\z(P_{\infty,1})
&=\int_{I_\infty}\psi_1=\frac{1}{1-\z^2}=\frac{\z+1}{3},\\
\jmath_\z(P_{\infty,2})&=\int_{\rho_\z\cdot I_\infty} \psi_1=
\int_{I_\infty} \rho_\z^*(\psi_1)
=\z\jmath_\z(P_{\infty,1})
=\frac{\z^2+\z}{3}
\equiv \frac{2\z+2}{3}\bmod L_\z.
\end{align*}
Note that 
$$
\jmath_\z(P_{0,1})=\int_{I_0}\psi_1
=\frac{1}{(1-\z^2)B\big(\frac{1}{3},\frac{1}{6}\big)}
\int_1^0 t^{1/2}(t-1)^{1/6}\frac{dt}{t(t-1)},
$$
$$
\int_1^0 t^{1/2}(t-1)^{1/6}\frac{dt}{t(t-1)}
=\ex\big(\frac{1}{12}\big)\int_0^1 t^{1/2}(1-t)^{1/6}\frac{dt}{t(1-t)}
=\ex\big(\frac{1}{12}\big) B\big(\frac{1}{2},\frac{1}{6}\big).
$$
Thus we have 
$$\jmath_\z(P_{0,1})=\frac{\ex(\frac{1}{12})}{1\!-\!\z^2}\cdot 
\frac{B\big(\frac{1}{2},\frac{1}{6}\big)}{B\big(\frac{1}{3},\frac{1}{6}\big)}
=\frac{(\z\!+\!1)\ex\big(\frac{1}{12}\big)}{3} \cdot
\frac{\G\big(\frac{1}{2}\big)\G\big(\frac{1}{2}\big)}
{\G\big(\frac{2}{3}\big)\G\big(\frac{1}{3}\big)}
=\frac{\sqrt{3}\ex\big(\frac{1}{6}\big)}{3}\cdot\frac{\sqrt{3}}{2}
=\frac{\z}{2}.$$
The rests are obtained as
$$\jmath_\z(P_{0,2})=\z\jmath_\z(P_{0,1})\equiv \frac{\z+1}{2}
\bmod L_\z, \quad 
\jmath_\z(P_{0,3})=\z^2\jmath_\z(P_{0,1})\equiv \frac{1}{2}\bmod L_\z,
$$
since 
$P_{0,2}=\rho_\z\cdot P_{0,1}$ and $P_{0,3}=\rho_\z^2\cdot P_{0,1}$.
\end{proof}

We consider the relation between the Abel-Jacobi map $\jmath_\z$ and 
the Schwarz map 
\begin{equation}
\label{eq:S(1/3,0,1/2)}
x\mapsto \frac{f_1(x)}{(1-\z^2)f_2(x)}=
\frac{2\sqrt{3}\z}
{B(\frac{1}{3},\frac{1}{6})}
\sqrt[6]{1-x}
F\big(\frac{1}{6},\frac{1}{2},\frac{7}{6};1-x)
\end{equation}
for $\CF\big(\frac{1}{3},0,\frac{1}{2}\big)$.
By Corollary \ref{cor:monod}, its monodromy group is generated by 
the three transformations
$$N_0:z\mapsto -z+\z,\quad N_1:z\mapsto \z z,\quad (N_0N_1)^{-1}:
z\mapsto \z^2z+1,$$
and this group is isomorphic to
the semi-direct product $\la \z\ra \ltimes \Z[\z]$. 
Note that the information of a branch of $u=\sqrt[6]{x^3(x-1)}$ is 
lost in the Schwarz map.  
Thus we can regard the Schwarz map as the Abel-Jacobi map $\jmath_\z$ 
modulo the actions of $\rho_\z$ and $\z$; that is 
$$C_\z/\la \rho_\z\ra \ni x\mapsto \int_1^x \psi_1 \in E_\z/\la \z\ra.$$

\subsection{The inverse of $\jmath_\z$}
We express the inverse of the Abel-Jacobi map $\jmath_\z$.
We regard the coordinates $t$ and $u$ as meromorphic functions on $C_\z$. 
The pull-backs ${\jmath_\z^{-1}}^*(t)$ and ${\jmath_\z^{-1}}^*(u)$ are 
elliptic functions with respect to the lattice $L_\z$, 
they can be expressed as 
$${\jmath_\z^{-1}}^*(t)=\theta_t(z),\quad {\jmath_\z^{-1}}^*(u)=\theta_u(z)$$
in terms of theta functions with characteristics. 
It turns out that the map 
$$E_\z\ni z \mapsto (\theta_t(z),\theta_u(z))\in C_\z$$
is  the inverse of $\jmath_\z$.

\begin{lemma}
\label{lem:2-ratio}
Let $z$ be the image of $(t,u)\in C_\z$ under the Abel-Jacobi map. 
Then we have 
$$
1+\frac{t}{u^2}= 
\sqrt{3}\ii\frac{\h_{0,0}(z,\z)^2}{\h_{\frac{1}{2},\frac{1}{2}}(z,\z)^2},\quad
1+\frac{\z^2 t}{u^2}=-\sqrt{3}  
\frac{\h_{\frac{1}{2},0}(z,\z)^2}{\h_{\frac{1}{2},\frac{1}{2}}(z,\z)^2},\quad
1+\frac{\z^4 t}{u^2}=\sqrt{3} 
\frac{\h_{0,\frac{1}{2}}(z,\z)^2}{\h_{\frac{1}{2},\frac{1}{2}}(z,\z)^2}.
$$
\end{lemma}
\begin{proof}
By Table \ref{tab:divisor}, we have 
$$1+\frac{t}{u^2}=c\cdot 
\frac{\h_{0,0}(z)^2}{\h_{\frac{1}{2},\frac{1}{2}}(z)^2},
$$
where $c$ is a constant. We substitute $P_{0,1}$ into the above, 
we have 
$$
1-\w=c\cdot \frac{\h_{0,0}(\z/2)^2}{\h_{\frac{1}{2},\frac{1}{2}}(\z/2)^2}
=c\cdot\Big(-\frac{\h_{\frac{1}{2},0}(0)^2}{\h_{0,\frac{1}{2}}(0)^2}\Big)
=-c\cdot \ex\big(\frac{1}{6}\big),
$$
which yields $c=\sqrt{3}\ii$.
The rests can be shown similarly.
%
%
\end{proof}

\begin{lemma}
\label{lem:0011-exp}
The functions $\h_{0,\frac{1}{2}}(z,\z)^2$ and $\h_{\frac{1}{2},0}(z,\z)^2$ 
are expressed as linear combinations of $\h_{0,0}(z,\z)^2$ and 
$\h_{\frac{1}{2},\frac{1}{2}}(z,\z)^2$:
\begin{align*}
\h_{0,\frac{1}{2}}(z,\z)^2&= \ex\big(\frac{-1}{12}\big)\left(
\h_{0,0}(z,\z)^2-\w^2\h_{\frac{1}{2},\frac{1}{2}}(z,\z)^2\right),\\
\h_{\frac{1}{2},0}(z,\z)^2&= \ex\big(\frac{1}{12}\big)\left(
\h_{0,0}(z,\z)^2+\w\h_{\frac{1}{2},\frac{1}{2}}(z,\z)^2\right).
\end{align*}
\end{lemma}
\begin{proof}
By Lemma \ref{lem:2-ratio}, we have 
$$\begin{array}{rlllc}
-\sqrt{3}\dfrac{\h_{\frac{1}{2},0}(z)^2}{\h_{\frac{1}{2},\frac{1}{2}}(z)^2}-1
&=\w\dfrac{t}{u^2}
&=
\w\left(\sqrt{3}\ii\dfrac{\h_{0,0}(z)^2}
{\h_{\frac{1}{2},\frac{1}{2}}(z)^2}-1\right),\\
\sqrt{3}\dfrac{\h_{0,\frac{1}{2}}(z)^2}
{\h_{\frac{1}{2},\frac{1}{2}}(z)^2}-1
&=\w^2\dfrac{t}{u^2}
&=
\w^2\left(\sqrt{3}\ii\dfrac{\h_{0,0}(z)^2}
{\h_{\frac{1}{2},\frac{1}{2}}(z)^2}-1\right),
\end{array}
$$
which yield this lemma.
\end{proof}

\begin{lemma}
\label{lem:u3/(t(t-1))}
Let $z$ be the image of $(t,u)\in C_\z$ under the Abel-Jacobi map. 
Then we have
$$
\frac{u^3}{t(t-1)}=
\ex\big(\frac{-1}{8}\big)\sqrt[4]{27}
\frac{\h_{0,0}(z,\z)\h_{0,\frac{1}{2}}(z,\z)
\h_{\frac{1}{2},0}(z,\z)}{\h_{\frac{1}{2},\frac{1}{2}}(z,\z)^3}.
$$
\end{lemma}

\begin{proof}
By Table \ref{tab:divisor}, we have 
$$
\frac{u^3}{t(t-1)}=c'\frac{\h_{0,0}(z)\h_{0,\frac{1}{2}}(z)
\h_{\frac{1}{2},0}(z)}{\h_{\frac{1}{2},\frac{1}{2}}(z)^3},
$$
where $c'$ is a constant.
We consider the limit as $t\to \infty$ with $t\in (1,\infty)$, 
$u\in (0,\infty)$. The left hand side of the above converges to $1$.
On the other hand, the right hand side of the above converges to 
\begin{align*}
& c'\frac{\h_{0,0}\big(\frac{\z+1}{3}\big)
\h_{0,\frac{1}{2}}\big(\frac{\z+1}{3}\big)
\h_{\frac{1}{2},0}\big(\frac{\z+1}{3}\big)}
{\h_{\frac{1}{2},\frac{1}{2}}\big(\frac{\z+1}{3}\big)^3}
=
c'\ex\Big(\frac{1}{3}\cdot\big(\frac{1}{2}+\frac{1}{2}\big)\Big)
\frac{\h_{\frac{1}{3},\frac{1}{3}}(0)\h_{\frac{1}{3},\frac{5}{6}}(0)
\h_{\frac{5}{6},\frac{1}{3}}\big(0)}
{\h_{\frac{5}{6},\frac{5}{6}}(0)^3}
\\
&=c'\ex\big(\frac{1}{3}-\frac{1}{8}-\frac{17}{24}+\frac{3}{6})
\frac{\h_{\frac{1}{3},\frac{1}{3}}(0)^3}
{\h_{\frac{1}{6},\frac{1}{6}}(0)^3}
=c'
\frac{\h_{\frac{1}{3},\frac{1}{3}}(0)^3}
{\h_{\frac{1}{6},\frac{1}{6}}(0)^3}=c'\ex\big(\frac{1}{8}\big)
\frac{1}{\sqrt[4]{27}}
\end{align*}
by Lemma \ref{lem:hi-theta}. Hence we have 
$c'=\ex\big(\frac{-1}{8}\big)\sqrt[4]{27}$.
\end{proof}

\begin{theorem}
\label{th:(t,u)-rep}
The inverse of $\jmath_\z:C_\z\ni (t,u)\mapsto z\in E_\z$ is given by 
\begin{align*}
t
&=
\frac{-3\sqrt{3}\ii
\h_{0,0}(z,\z)^2\h_{0,\frac{1}{2}}(z,\z)^2\h_{\frac{1}{2},0}(z,\z)^2}
{(\sqrt{3}\ii\h_{0,0}(z,\z)^2-\h_{\frac{1}{2},\frac{1}{2}}(z,\z)^2)^3},\\
u&=\ex\big(\frac{-1}{8}\big)\sqrt[4]{27}
\frac{\h_{0,0}(z,\z)\h_{0,\frac{1}{2}}(z,\z)
\h_{\frac{1}{2},0}(z,\z)\h_{\frac{1}{2},\frac{1}{2}}(z,\z)}
{\big(\sqrt{3}\ii\h_{0,0}(z,\z)^2-
\h_{\frac{1}{2},\frac{1}{2}}(z,\z)^2\big)^2}.
\end{align*}
\end{theorem}
\begin{proof}
Note that 
$$\big(1+\frac{t}{u^2}\big)\big(1+\frac{\z^2t}{u^2}\big)
\big(1+\frac{\z^4t}{u^2}\big)=1+\frac{t^3}{u^6}=1+\frac{1}{t-1}.$$
By Lemma \ref{lem:2-ratio}, we have 
$$1+\frac{1}{t-1}=-3\sqrt{3}\ii
\frac{\h_{0,0}(z)^2
\h_{0,\frac{1}{2}}(z)^2
\h_{\frac{1}{2},0}(z)^2}
{\h_{\frac{1}{2},\frac{1}{2}}(z)^6},
$$
which yields 
$$t=\frac{3\sqrt{3}\ii
\h_{0,0}(z)^2\h_{0,\frac{1}{2}}(z)^2\h_{\frac{1}{2},0}(z)^2}
{3\sqrt{3}\ii
\h_{0,0}(z)^2\h_{0,\frac{1}{2}}(z)^2\h_{\frac{1}{2},0}(z)^2
+\h_{\frac{1}{2},\frac{1}{2}}(z)^6}.$$
Rewrite $\h_{0,\frac{1}{2}}(z)^2$ and $\h_{\frac{1}{2},0}(z)^2$ 
in the denominator of this expression 
by $\h_{0,0}(z)^2$ and $\h_{\frac{1}{2},\frac{1}{2}}(z)^2$  
by Lemma \ref{lem:0011-exp}. 
Then it can be factorized as
$$
 3\sqrt{3}\ii 
\h_{0,0}(z)^2\h_{0,\frac{1}{2}}(z)^2\h_{\frac{1}{2},0}(z)^2
+\h_{\frac{1}{2},\frac{1}{2}}(z)^6\\
=-(\sqrt{3}\ii\h_{0,0}(z)^2-\h_{\frac{1}{2},\frac{1}{2}}(z)^2)^3.
$$
Hence we have the expression of $t$.

By Lemmas \ref{lem:2-ratio} and \ref{lem:u3/(t(t-1))}, the functions 
$1+\dfrac{t}{u^2}$ and $\dfrac{u^3}{t(t-1)}$ are expressed in terms 
$\h_{a,b}(z,\z)$.
We have 
\begin{align*}
 u&=\frac{u^3}{t(t-1)}\cdot 
\left(\big(1+\frac{t}{u^2}\big)-1\right)\cdot (t-1)\\
&=\ex\big(\frac{-1}{8}\big)\sqrt[4]{27}
\frac{\h_{0,0}(z)\h_{0,\frac{1}{2}}(z)\h_{\frac{1}{2},0}(z)}
{\h_{\frac{1}{2},\frac{1}{2}}(z)^3}
\cdot \frac{\sqrt{3}\ii\h_{0,0}(z)^2
-\h_{\frac{1}{2},\frac{1}{2}}(z)^2}
{\h_{\frac{1}{2},\frac{1}{2}}(z)^2}\\
&\hspace{5mm} \cdot 
\frac{-\h_{\frac{1}{2},\frac{1}{2}}(z)^6}
{3\sqrt{3}\ii
\h_{0,0}(z)^2\h_{0,\frac{1}{2}}(z)^2\h_{\frac{1}{2},0}(z)^2
+\h_{\frac{1}{2},\frac{1}{2}}(z)^6}\\
&=\ex\big(\frac{-1}{8}\big)\sqrt[4]{27}
\frac{\h_{0,0}(z)\h_{0,\frac{1}{2}}(z)
\h_{\frac{1}{2},0}(z)\h_{\frac{1}{2},\frac{1}{2}}(z)
(\sqrt{3}\ii\h_{0,0}(z)^2-\h_{\frac{1}{2},\frac{1}{2}}(z)^2)}
{-3\sqrt{3}\ii 
\h_{0,0}(z)^2\h_{0,\frac{1}{2}}(z)^2
\h_{\frac{1}{2},0}(z)^2-\h_{\frac{1}{2},\frac{1}{2}}(z)^6}.
\end{align*}
Note that the denominator of the last term is 
$(\sqrt{3}\ii\h_{0,0}(z)^2-\h_{\frac{1}{2},\frac{1}{2}}(z)^2)^3$.
Hence we have the expression of $u$. 
\end{proof}

\begin{cor}
\label{cor:1-form}
The pull back of the holomorphic $1$-from $\psi=\dfrac{udt}{t(t-1)}$ 
under the map $\jmath_\z^{-1}$ is 
$$
\ex\big(\frac{-1}{8}\big)2\pi\sqrt[4]{27}\h_{0,0}(0,\z)^2dz.
$$
The theta constant $\h_{0,0}(0,\z)$ is evaluated as
$$
\h_{0,0}(0,\z)=\ex\big(\dfrac{1}{48}\big)\dfrac{\sqrt[8]{3}}{\sqrt[3]{4}\pi}
\G\big(\dfrac{1}{3}\big)^{3/2}.
$$
The other theta constants $\h_{a,b}(0,\z)$ are
$$\begin{array}{ll}
\h_{0,\frac{1}{2}}(0,\z)
=\ex\big(\dfrac{-1}{48}\big)\dfrac{\sqrt[8]{3}}{\sqrt[3]{4}\pi}
\G\big(\dfrac{1}{3}\big)^{3/2},&
\h_{\frac{1}{2},0}(0,\z)
=\ex\big(\dfrac{1}{16}\big)\dfrac{\sqrt[8]{3}}{\sqrt[3]{4}\pi}
\G\big(\dfrac{1}{3}\big)^{3/2},\\[4mm]
\h_{\frac{1}{3},\frac{1}{3}}(0,\z)
=\ex\big(\dfrac{11}{144}\big)\dfrac{\sqrt[8]{3}}{2\pi}
\G\big(\dfrac{1}{3}\big)^{3/2},&
\h_{\frac{1}{6},\frac{1}{6}}(0,\z)
=\ex\big(\dfrac{5}{144}\big)\dfrac{\sqrt[8]{27}}{2\pi}
\G\big(\dfrac{1}{3}\big)^{3/2}.\\[4mm]
\h_{\frac{5}{6},\frac{1}{3}}(0,\z)
=\ex\big(\dfrac{-7}{144}\big)\dfrac{\sqrt[8]{3}}{2\pi}
\G\big(\dfrac{1}{3}\big)^{3/2},&
\h_{\frac{1}{3},\frac{5}{6}}(0,\z)
=\ex\big(\dfrac{53}{144}\big)\dfrac{\sqrt[8]{3}}{2\pi}
\G\big(\dfrac{1}{3}\big)^{3/2}.
\end{array}
$$
\end{cor}
\begin{proof}
Recall that 
$$\frac{t}{t-1}=-3\sqrt{3}\ii \frac{\h_{0,0}(z)^2
\h_{0,\frac{1}{2}}(z)^2\h_{\frac{1}{2},0}(z)^2}
{\h_{\frac{1}{2},\frac{1}{2}}(z)^6}.
$$
Thus we have 
\begin{align*}
& \frac{dt}{t^2}=d\Big(1-\frac{1}{t}\Big)=d\Big(\frac{t-1}{t}\Big)\\
&=\frac{dz}{-3\sqrt{3}\ii}\Bigg[
\frac{
6\h_{\frac{1}{2},\frac{1}{2}}(z)^5\h_{\frac{1}{2},\frac{1}{2}}(z)'\cdot 
\h_{0,0}(z)^2\h_{0,\frac{1}{2}}(z)^2\h_{\frac{1}{2},0}(z)^2}
{\h_{0,0}(z)^4\h_{0,\frac{1}{2}}(z)^4\h_{\frac{1}{2},0}(z)^4}\\
& \hspace{25mm}-\frac{\h_{\frac{1}{2},\frac{1}{2}}(z)^6\cdot
(\h_{0,0}(z)^2\h_{0,\frac{1}{2}}(z)^2\h_{\frac{1}{2},0}(z)^2)'}
{\h_{0,0}(z)^4\h_{0,\frac{1}{2}}(z)^4\h_{\frac{1}{2},0}(z)^4}\Bigg],
\end{align*}
where $f(z)'= \dfrac{df(z)}{dz}$.
Since 
$\psi=u\cdot \dfrac{t}{t-1} \cdot \dfrac{dt}{t^2}$, the pull-back of $\psi$ 
under the map $\jmath_\z^{-1}$ is $\ex\big(\frac{-1}{8}\big)\sqrt[4]{27}$ 
times 
$$
\frac{6\h_{0,0}(z)^2\h_{0,\frac{1}{2}}(z)^2
\h_{\frac{1}{2},0}(z)^2
\h_{\frac{1}{2},\frac{1}{2}}(z)'-
\h_{\frac{1}{2},\frac{1}{2}}(z)(\h_{0,0}(z)^2\h_{0,\frac{1}{2}}(z)^2
\h_{\frac{1}{2},0}(z)^2)'}{(\sqrt{3}\ii 
\h_{0,0}(z)^2-\h_{\frac{1}{2},\frac{1}{2}}(z)^2)^2
\h_{0,0}(z)\h_{0,\frac{1}{2}}(z)
\h_{\frac{1}{2},0}(z)
}dz.
$$
It should be a constant times $dz$. We determine this constant by 
substituting $z=0$ into the above. By Fact \ref{fact:derivative}  
and Lemma \ref{lem:hi-theta}, we have 
$${\jmath_\z^{-1}}^*(\psi)
=\ex\big(\frac{-1}{8}\big)2\pi\sqrt[4]{27}\h_{0,0}(0,\z)^2dz.$$
Note that 
$$B\big(\frac{1}{3},\frac{1}{6}\big)
=\int_1^\infty \psi=\int_{\jmath_\z(P_1)}^{\jmath_\z(P_{\infty,1})} {\jmath_\z^{-1}}^*
(\psi)=
\ex\big(\frac{-1}{8}\big)2\pi\sqrt[4]{27}\h_{0,0}(0,\z)^2
\cdot\Big(\frac{\z+1}{3}-0\Big)
$$
by Proposition \ref{prop:jmath-image}.
The  well-known formula 
$$\G\big(\frac{1}{6}\big)=\frac{1}{\sqrt[3]{2}}\frac{\sqrt{3}}{\sqrt{\pi}}
\G\big(\frac{1}{3}\big)^2$$
yields that 
$$B\big(\frac{1}{3},\frac{1}{6}\big)=
\frac{\G\big(\frac{1}{3}\big)\G\big(\frac{1}{6}\big)}
{\G\big(\frac{1}{2}\big)}=\frac{\sqrt{3}}{\sqrt[3]{2}\pi}
\G\big(\frac{1}{3}\big)^3.
$$
Hence we evaluate the theta constant as 
$$\h_{0,0}(0,\z)^2=
\ex\big(\frac{1}{24}\big)\frac{\sqrt[4]{3}}{\sqrt[3]{16}\pi^2}
\G\big(\frac{1}{3}\big)^3.$$
We can determine the sign of $\h_{0,0}(0,\z)$ by a numerical computation. 
The rests can be obtained by Lemma \ref{lem:hi-theta}.
\end{proof}

\begin{cor}
\label{cor:inv-Schwarz2}
The inverse of the Schwarz map (\ref{eq:S(1/3,0,1/2)}) 
for $\CF(\frac{1}{3},0,\frac{1}{2})$ 
is given by
$$
x=\frac{-3\sqrt{3}\ii
\h_{0,0}(z,\z)^2\h_{0,\frac{1}{2}}(z,\z)^2 \h_{\frac{1}{2},0}(z,\z)^2}
{(\sqrt{3}\ii\h_{0,0}(z,\z)^2-\h_{\frac{1}{2},\frac{1}{2}}(z,\z)^2)^3}.
$$
\end{cor}
\begin{proof}
It is clear by Theorem \ref{th:(t,u)-rep}. We can check this map 
is invariant under the action of $\la \z\ra$ by Lemma \ref{lem:omega-times}.
\end{proof}

\begin{cor} For any point $z$ around $0$, we have 
\label{cor:HGF-Theta2}
$$\frac{\sqrt[3]{16}\pi\z^2}{\G\big(\frac{1}{3}\big)^3}\cdot 
\frac{1}
{\sqrt{1-\sqrt{3}\ii\dfrac{\h_{0,0}(z,\z)^2}
{\h_{\frac{1}{2},\frac{1}{2}}(z,\z)^2}}}
\cdot 
F\big(\frac{1}{6},\frac{1}{2},\frac{7}{6};
\frac{\h_{\frac{1}{2},\frac{1}{2}}(z,\z)^6}
{(\h_{\frac{1}{2},\frac{1}{2}}(z,\z)^2-\sqrt{3}\ii\h_{0,0}(z,\z)^2)^3}\big)
=z,
$$
where the branch of the square root is selected as 
$\sqrt{\z^2}=\z$ for $z=\dfrac{\z}{2}$.
\end{cor}
\begin{proof}
Let $z$ be the image of the Schwarz map (\ref{eq:S(1/3,0,1/2)}). 
We have seen in Proof of Theorem \ref{th:(t,u)-rep} that 
$$\frac{1}{1-x}
=\frac{3\sqrt{3}\ii
\h_{0,0}(z)^2\h_{0,\frac{1}{2}}(z)^2\h_{\frac{1}{2},0}(z)^2+
\h_{\frac{1}{2},\frac{1}{2}}(z)^6}
{\h_{\frac{1}{2},\frac{1}{2}}(z)^6}
=\frac{(\h_{\frac{1}{2},\frac{1}{2}}(z)^2-\sqrt{3}\ii\h_{0,0}(z)^2)^3}
{\h_{\frac{1}{2},\frac{1}{2}}(z)^6}.$$
Thus we have the desired identity 
modulo the monodromy group of $\CF\big(\frac{1}{3},0,\frac{1}{2}\big)$.
Since the both sides of the above become $0$ for $z=0$, 
their difference is represented as the group $\la \z\ra$.
Consider the limit of the both sides as $z\to \dfrac{\z}{2}$ along 
the segment connecting $0$ and $\dfrac{\z}{2}$. 
Since $1/(1-x)$ converges to $1$ by this limit, 
it turns out that $x$ converges to $0$.
Use
$$ 
1-\sqrt{3}\ii\frac{\h_{0,0}(\frac{\z}{2},\z)^2}
{\h_{\frac{1}{2},\frac{1}{2}}(\frac{\z}{2},\z)^2}
=
1+\sqrt{3}\ii\frac{\h_{\frac{1}{2},0}(0,\z)^2}
{\h_{0,\frac{1}{2}}(0,\z)^2}=1+\sqrt{3}\ii\z=\z^2
$$
and the Gauss-Kummer formula.
\end{proof}


\subsection{$(1+\z)$-multiplication}
\begin{theorem}
\label{th:(1+z)-multiple}
Let $z\in E_\z$ be the image of $(t,u)\in C_\z$ under the Abel-Jacobi map 
$\jmath_\z$. Then we have
\begin{equation}
\label{eq:(1+z)-mul}
\jmath_\z^{-1}((1+\z)z) =\left(\frac{t(9-8t)^2}{(4t-3)^3},
\ex\big(\frac{1}{12}\big)\sqrt{3}u\frac{9-8t}{(4t-3)^2}
\right).
\end{equation}
\end{theorem}

\begin{proof}
We set $(t',u')=\jmath_\z^{-1}((1+\z)z)$.
Then $t'$ is given by 
the substitution $z$ to $(z+1)z$ into 
the expression of $t$ in Theorem \ref{th:(t,u)-rep}. 
Rewrite $\h_{a,b}((1+\z)z)$ in terms of $\h_{a,b}(z)$ by 
Lemma \ref{lem:(z+1)-times}. Its numerator $N(t')$ is 
\begin{align*}
N(t')=&3\sqrt{3}\ii\h_{0,0}(z)^2\h_{0,\frac{1}{2}}(z)^2
\h_{\frac{1}{2},0}(z)^2\\
&\  \times(\h_{0,0}(z)^2-\ii\h_{0,\frac{1}{2}}(z)^2)^2
(\h_{0,0}(z)^2+\ii\h_{\frac{1}{2},0}(z)^2)^2
(\h_{0,\frac{1}{2}}(z)^2-\h_{\frac{1}{2},0}(z)^2)^2,
\end{align*}
and its denominator $D(t')$ is 
\begin{align*}
D(t')=&\{-(\sqrt{3}\h_{\frac{1}{2},0}(z)^2+
\h_{\frac{1}{2},\frac{1}{2}}(z)^2)
(\h_{0,0}(z)^4-\h_{0,\frac{1}{2}}(z)^4)\\
&\  +2\ii(\sqrt{3}\h_{\frac{1}{2},0}(z)^2
-\h_{\frac{1}{2},\frac{1}{2}}(z)^2)\h_{0,0}(z)^2
\h_{0,\frac{1}{2}}(z)^2\}^3.
\end{align*}
In this computation, 
the theta constants $\h_{0,0}(0)$, $\h_{0,\frac{1}{2}}(0)$, 
$\h_{\frac{1}{2},0}(0)$ are canceled 
by Lemma \ref{lem:hi-theta}. 
Divide them by $\h_{\frac{1}{2},\frac{1}{2}}(z)^{18}$ and 
rewrite 
$$\frac{\h_{0,0}(z)^2}
{\h_{\frac{1}{2},\frac{1}{2}}(z)^2}=\frac{1+\dfrac{t}{u^2}}{\sqrt{3}\ii},
\quad 
\frac{\h_{\frac{1}{2},0}(z)^2}
{\h_{\frac{1}{2},\frac{1}{2}}(z)^2}=\frac{1+\dfrac{\w t}{u^2}}{-\sqrt{3}},
\quad 
\frac{\h_{0,\frac{1}{2}}(z)^2}
{\h_{\frac{1}{2},\frac{1}{2}}(z)^2}=\frac{1+\dfrac{\w^2 t}{u^2}}{\sqrt{3}}.
$$
Then we have
$$t'=\left(\frac{-(t^3+u^6)(t^3-8u^6)^2}{27u^{18}}\right)\Bigg/
\left(\frac{-(t^3+4u^6)^3}{27u^{18}}\right)
=\frac{t(9-8t)^2}{(4t-3)^3},
$$
where we use the relation $u^6=t^3(t-1)$.

By the same way, we can express $u'$ in terms of $\h_{a,b}(z)$'s, 
whose numerator $N(u')$ and denominator $D(u')$ are 
\begin{align*}
N(u')=&\ex\big(\frac{1}{8}\big)
\sqrt[4]{27}\h_{0,0}(z)\h_{0,\frac{1}{2}}(z)\h_{\frac{1}{2},0}(z)
\h_{\frac{1}{2},\frac{1}{2}}(z)
\big(\h_{0,\frac{1}{2}}(z)^2-\h_{\frac{1}{2},0}(z)^2\big)\\
&\ \cdot\big(\h_{0,0}(z)^4+\h_{0,\frac{1}{2}}(z)^4\big)
\big(\h_{0,0}(z)^2+\ii\h_{\frac{1}{2},0}(z)^2\big)
,\\
D(u')=&\{(\sqrt{3}\h_{\frac{1}{2},0}(z)^2+
\h_{\frac{1}{2},\frac{1}{2}}(z)^2)
(\h_{0,0}(z)^4-\h_{0,\frac{1}{2}}(z)^4)\\
&\ -2\ii(\sqrt{3}\h_{\frac{1}{2},0}(z)^2
-\h_{\frac{1}{2},\frac{1}{2}}(z)^2)\h_{0,0}(z)^2
\h_{0,\frac{1}{2}}(z)^2\}^2.
\end{align*}
Divide them by $\h_{\frac{1}{2},\frac{1}{2}}(z)^8(\sqrt{3}\ii
\h_{0,0}(z)^2-\h_{\frac{1}{2},\frac{1}{2}}(z)^2)^2$. 
We factor out $u$ from the numerator as 
\begin{align*}
& \frac{N(u')}
{\h_{\frac{1}{2},\frac{1}{2}}(z)^8
(\sqrt{3}\ii\h_{0,0}(z)^2-\h_{\frac{1}{2},\frac{1}{2}}(z)^2)^2}\\
&=\ii u \frac{
\big(\h_{0,\frac{1}{2}}(z)^2-\h_{\frac{1}{2},0}(z)^2\big)
\big(\h_{0,0}(z)^4+\h_{0,\frac{1}{2}}(z)^4\big)
\big(\h_{0,0}(z)^2+\ii\h_{\frac{1}{2},0}(z)^2\big)}
{\h_{\frac{1}{2},\frac{1}{2}}(z)^8}.
\end{align*}
Since the rest terms are expressed in terms of $\h_{a,b}(z)^2$, 
we can compute them quite similarly to the case of $t'$.
Hence we have 
$$u'=\ii u\left(\frac{(3\ii-\sqrt{3})(-t^3+8u^6)t}{18u^8}\right)\Bigg/
\left(\frac{(t^3+4u^6)^2}{9t^2u^8}\right)=
\frac{3+\sqrt{3}\ii}{2}\cdot u\cdot\frac{(9-8t)}{(4t-3)^2}.$$
It is easy to see that $(t',u')$ satisfies ${u'}^6={t'}^3(t'-1)$. 
\end{proof}

\section{Limits of mean iterations}
\subsection{Limit formula by $F(\frac{1}{4},\frac{1}{2},\frac{5}{4};x)$}
Theorem \ref{th:(1+i)-multiple} is interpreted as follows.
\begin{theorem}
\label{th:interpret}
Let 
$P_x=(x,\sqrt[4]{x^2(x-1)})$ be a point of the curve $C$.
We set 
$$P_{x'}=\left(\frac{(2-x)^2}{x^2}, 
\frac{(1+\ii)(2-x)\sqrt[4]{x^2(x-1)}}{x^2}\right)\in C.$$
Then we have
$$\int_{P_1}^{P_{x'}}\f\equiv (1+\ii)\int_{P_1}^{P_x}\f\quad \bmod \la \ii\ra
\ltimes \Z[\ii].
$$
\end{theorem}

\begin{cor}
\label{cor:eq-HGF}
The following identity holds around $x=1$:
$$\frac{1}{\sqrt{x}}
F(\frac{1}{4},\frac{1}{2},\frac{5}{4},1-\frac{(2-x)^2}{x^2})
=F(\frac{1}{4},\frac{1}{2},\frac{5}{4},1-x).
$$
\end{cor}

\begin{proof}
Theorem \ref{th:interpret} implies that 
$$
\int_1^{(2-x)^2/x^2} \frac{\sqrt[4]{t^2(t-1)}dt}{t(t-1)}
\equiv(1+\ii)\int_1^{x} \frac{\sqrt[4]{t^2(t-1)}dt}{t(t-1)}\quad 
\bmod \la \ii\ra \ltimes \Z[\ii].
$$
Note that
$$
\int_1^{x} \frac{\sqrt[4]{t^2(t-1)}dt}{t(t-1)}
=2\sqrt{2}(1+\ii)\sqrt[4]{1-x}F(\frac{1}{4},\frac{1}{2},\frac{5}{4},1-x),
$$
$$\sqrt[4]{1-\frac{(2-x)^2}{x^2}}=\sqrt[4]{\frac{4x-4}{x^2}}
=(1+\ii)\frac{\sqrt[4]{1-x}}{\sqrt{x}},
$$
for $0<x<1$ and $\arg(1-(2-x)^2/x^2)=\pi$. 
We can cancel the factor $\sqrt[4]{1-x}$ and determine the action of 
$\la \ii\ra \ltimes \Z[\ii]$ by the substitution $x=1$. Thus 
we have the desired identity.
\end{proof}

Let $a=a_1$ and $b=b_1$ be positive real numbers.
We define a pair $\{a_n,b_n\}_{n\in \N}$ of sequences by 
the recursive relations
\begin{equation}
\label{eq:seq1}
a_{n+1}=\frac{a_n+b_n}{2},\quad b_{n+1}=\sqrt{\frac{a_n(a_n+b_n)}{2}}.
\end{equation}


\begin{cor}[A formula in Theorem 2 in \cite{HKM}]
\label{cor:AGM1}
We have
$$\lim_{n\to\infty} a_n=\lim_{n\to\infty} b_n
=\frac{a}{F(\frac{1}{4},\frac{1}{2},\frac{5}{4};1-\frac{b^2}{a^2})^2}.$$
\end{cor}
\begin{proof}
We can show that the sequences $\{a_n\}$ and $\{b_n\}$  converge and 
$\lim\limits_{n\to \infty} a_n=\lim\limits_{n\to \infty} b_n$
by Lemma 1 in \cite{HKM}.
Substitute $x=2a/(a+b)$ into the identity between hypergeometric series 
in Corollary \ref{cor:eq-HGF}.
Since 
$$\frac{2-x}{x}=\frac{b}{a},$$
we have 
$$
\frac{\sqrt{a+b}}{\sqrt{2a}}
F(\frac{1}{4},\frac{1}{2},\frac{5}{4};1-\frac{b^2}{a^2})
=F(\frac{1}{4},\frac{1}{2},\frac{5}{4};1-\frac{2a}{a+b}),$$
\begin{align*}
& \frac{a}{F(\frac{1}{4},\frac{1}{2},\frac{5}{4};1-\frac{b^2}{a^2})^2}
=\frac{(a+b)/2}
{F\left(\frac{1}{4},\frac{1}{2},\frac{5}{4};1-
  \left({\frac{\sqrt{a(a+b)/2}}
{(a+b)/2}}\right)^2\right)^2}\\
&=\frac{a_2}{F(\frac{1}{4},\frac{1}{2},\frac{5}{4};1-\frac{b_2^2}{a_2^2})^2}=
\cdots
=\frac{a_n}{F(\frac{1}{4},\frac{1}{2},\frac{5}{4};1-\frac{b_n^2}{a_n^2})^2}.
\end{align*}
The last term is equal to  $\lim\limits_{n\to \infty} a_n$ 
since $\lim\limits_{n\to\infty}\dfrac{b_n^2}{a_n^2}=1$ and 
$F(\frac{1}{4},\frac{1}{2},\frac{5}{4};0)=1$. 
\end{proof}

Hence we see that the $(1+\ii)$-multiple formula (\ref{eq:(1+i)-mul}) 
in Theorem \ref{th:(1+i)-multiple} implies 
this limit formula for the sequences defined by 
the mean iteration (\ref{eq:seq1}).

\subsection{Limit formula by $F(\frac{1}{6},\frac{1}{2},\frac{7}{6};x)$}
Theorem \ref{th:(1+z)-multiple} is interpreted as follows.
\begin{theorem}
\label{th:interpret2}
Let 
$P_x=(x,\sqrt[6]{x^3(x-1)})$ be a point of the curve $C_\z$.
We set 
$$P_{x'}=
\left(\frac{x(9-8x)^2}{(4x-3)^3},
\ex\big(\frac{1}{12}\big)\sqrt{3}\sqrt[6]{x^3(x-1)}\frac{9-8x}{(4x-3)^2}
\right)\in C_\z.$$
Then we have
$$\int_{P_1}^{P_{x'}}\psi\equiv(1+\z)\int_{P_1}^{P_x}\psi\quad \mod \la \z\ra
\ltimes \Z[\w].
$$
\end{theorem}

\begin{cor} The following identity holds around $x=1$:
\label{cor:(1+z)-times}
$$F\big(\frac{1}{6},\frac{1}{2},\frac{7}{6};1-x\big)
=\frac{1}{\sqrt{4x-3}}F\big(\frac{1}{6},\frac{1}{2},\frac{7}{6};
1-\frac{x(9-8x)^2}{(4x-3)^3}\big),
$$
where $\sqrt{4x-3}=1$ for $x=1$.
\end{cor}
\begin{proof}
Theorem \ref{th:interpret2} implies that 
$$
\int_1^{x'} \frac{\sqrt[6]{t^3(t-1)}dt}{t(t-1)}
\equiv(1+\z)\int_1^{x} \frac{\sqrt[6]{t^3(t-1)}dt}{t(t-1)}\quad \mod
\la \z\ra\ltimes \Z[\w], \quad
$$
for $x'=\dfrac{x(9-8x)^2}{(4x-3)^3}.$
By this relation, there exists $k\in \N$ such that 
$$
\z^k\frac{\sqrt[6]{27(x-1)}}{\sqrt{4x-3}}
F\big(\frac{1}{6},\frac{1}{2},\frac{7}{6};1-\frac{x(9-8x)^2}{(4x-3)^3}\big)
=(1+\z)\sqrt[6]{1-x}F\big(\frac{1}{6},\frac{1}{2},\frac{7}{6};1-x\big).
$$
We cancel $(1+\z)\sqrt[6]{1-x}$ and $\sqrt[6]{27(x-1)}$,  
and choose $k=0$ so that the identity holds for $x=1$.
\end{proof}

By Corollary \ref{cor:(1+z)-times}, we define two means as follows. 
We solve the cubic equation 
$$\dfrac{x(9-8x)^2}{(4x-3)^3}=\frac{b^2}{a^2}$$
of the variable $x$, where we assume $0<a<b$.
 A real solution $x_0$ of this equation is 
$$\frac{3}{8} \Bigg[ \frac{\sqrt[3]{a^2}}{\sqrt{b^2-a^2}}
\Big(\sqrt[3]{b+\sqrt{b^2-a^2}}-\sqrt[3]{b-\sqrt{b^2-a^2}}\Big)+2\Bigg].$$
We set 
\begin{equation}
\label{eq:preimage}
\eta_1=b+\sqrt{b^2-a^2},\quad \eta_2=b-\sqrt{b^2-a^2}.
\end{equation}
Note that
$$\eta_1\eta_2=a^2,\quad \frac{\eta_1+\eta_2}{2}=b,\quad 
\frac{\eta_1-\eta_2}{2}=\sqrt{b^2-a^2}.$$
We express $x_0$ and $4x_0-3$ in terms of $\eta_1$ and $\eta_2$ as 
\begin{align*}
x_0&= \frac{3}{8}
\Bigg[\frac{\sqrt[3]{\eta_1}\sqrt[3]{\eta_2}}{(\eta_1-\eta_2)/2}
\Big(\sqrt[3]{\eta_1}-\sqrt[3]{\eta_2}\Big)+2\Bigg]
=\frac{3}{4}
\Big[\frac{\sqrt[3]{\eta_1}\sqrt[3]{\eta_2}}
{\sqrt[3]{\eta_1^2}+\sqrt[3]{\eta_1}\sqrt[3]{\eta_2}+\sqrt[3]{\eta_2^2}}
+1\Big]\\
&=\frac{3}{4}
\frac{(\sqrt[3]{\eta_1}+\sqrt[3]{\eta_2})^2}
{\sqrt[3]{\eta_1^2}+\sqrt[3]{\eta_1}\sqrt[3]{\eta_2}+\sqrt[3]{\eta_2^2}},\\
4x_0-3&=
\frac{3}{2}
\Bigg[\frac{\sqrt[3]{\eta_1}\sqrt[3]{\eta_2}}{(\eta_1-\eta_2)/2}
\Big(\sqrt[3]{\eta_1}-\sqrt[3]{\eta_2}\Big)+2\Bigg]-3
=\frac{3\sqrt[3]{\eta_1}\sqrt[3]{\eta_2}}
{\sqrt[3]{\eta_1^2}+\sqrt[3]{\eta_1}\sqrt[3]{\eta_2}+\sqrt[3]{\eta_2^2}}.
\end{align*}
Thus the identity in Corollary \ref{cor:(1+z)-times} is 
transformed into 
\begin{align*}
& F\Big(\frac{1}{6},\frac{1}{2},\frac{7}{6};1-
\big(\frac{\sqrt[3]{\eta_1}+\sqrt[3]{\eta_2}}{2}
\big)^2\Big/\big(\sqrt{
\frac{\eta_1^{2/3}+\eta_1^{1/3}\eta_2^{1/3}+\eta_2^{2/3}}{3}}
\big)^2\Big)\\
&=
\frac{1}{\sqrt[3]{a}}\sqrt{
\frac{\eta_1^{2/3}+\eta_1^{1/3}\eta_2^{1/3}+\eta_2^{2/3}}{3}}
F\Big(\frac{1}{6},\frac{1}{2},\frac{7}{6};1-\frac{b^2}{a^2}\Big).
\end{align*}
This formula is equivalent to 
\begin{equation}
\label{eq:recurrent}
\frac{a}{F\Big(\frac{1}{6},\frac{1}{2},\frac{7}{6};1-\frac{b^2}{a^2}\Big)}
=
\frac{m_1(a,b)}{F\Big(\frac{1}{6},\frac{1}{2},\frac{7}{6};
1-\frac{m_2(a,b)^2}{m_1(a,b)^2}\Big)}
\end{equation}
if we define two means $m_1$ and $m_2$ of positive real numbers $a$ and $b$ by 
$$
m_1(a,b)=
\frac{
a^{2/3}\sqrt{\eta_1^{2/3}+\eta_1^{1/3}\eta_2^{1/3}+\eta_2^{2/3}}}
{\sqrt{3}},\quad 
m_2(a,b)=\frac{a^{2/3}(\eta_1^{1/3}+\eta_2^{1/3})}{2},
$$
where $\eta_1$ and $\eta_2$ are given in (\ref{eq:preimage}) with conditions 
$$
-\frac{\pi}{6}<\arg(\eta_i^{1/3})<\frac{\pi}{6},\quad 
\eta_1^{1/3}\eta_2^{1/3}=a^{2/3}.$$
Let $a_1=a$ and $b_1=b$ be positive real numbers.  
We give a pair of sequences $\{a_n,b_n\}_{n\in \N}$ 
with initial terms $a_1=a$, $b_1=b$ by the recursive relations
\begin{equation}
\label{eq:seq2}
a_{n+1}=m_1(a_n,b_n),\quad b_{n+1}=m_2(a_n,b_n).
\end{equation}

\begin{cor}[A formula in Theorem 3 in \cite{HKM}]
\label{cor:AGM2}
We have 
$$
\lim_{n\to \infty} a_n=\lim_{n\to \infty} b_n
=\dfrac{a}{F\big(\frac{1}{6},\frac{1}{2},\frac{7}{6};
1-\frac{b^2}{a^2}\big)}.
$$
\end{cor}
\begin{proof}
It is shown in \S5 of \cite{HKM} that the sequences $\{a_n\}$ and $\{b_n\}$ 
converge and satisfy 
$\lim\limits_{n\to\infty}a_n=\lim\limits_{n\to\infty}b_n$. 
By (\ref{eq:recurrent}), we have 
\begin{align*}
& \frac{a}{F\Big(\frac{1}{6},\frac{1}{2},\frac{7}{6};1-\frac{b^2}{a^2}\Big)}
=\frac{a_2}{F\Big(\frac{1}{6},\frac{1}{2},\frac{7}{6};
1-\frac{b_2^2}{a_2^2}\Big)}
=\frac{a_3}{F\Big(\frac{1}{6},\frac{1}{2},\frac{7}{6};
1-\frac{b_3^2}{a_3^2}\Big)}\\
&=\cdots
=\frac{a_n}{F\Big(\frac{1}{6},\frac{1}{2},\frac{7}{6};
1-\frac{b_n^2}{a_n^2}\Big)}=\cdots =\lim_{n\to \infty} a_n,
\end{align*}
since 
$\lim\limits_{n\to \infty}\dfrac{b_n^2}{a_n^2}=1$ and 
$F\Big(\frac{1}{6},\frac{1}{2},\frac{7}{6};0\Big)=1.$
\end{proof}
Hence we see that the $(1+\z)$-multiple formula (\ref{eq:(1+z)-mul}) 
in Theorem \ref{th:(1+z)-multiple} 
implies this limit formula for the sequences defined by 
the mean iteration (\ref{eq:seq2}).

\end{document}